\documentclass[11pt,a4paper]{amsart}
\usepackage[all]{xy}
\SelectTips{cm}{}
\calclayout

\usepackage{amscd,mathtools,amssymb,amsopn,amsmath,amsthm,graphics,mathrsfs,accents}
\usepackage[colorlinks=true,linkcolor=red,citecolor=blue]{hyperref}
\usepackage[mathscr]{euscript}
\usepackage[usenames]{color}

\newcommand{\rt}{\rightarrow}
\newcommand{\rat}{\rightarrowtail}

\newcommand{\re}{\twoheadrightarrow}
\newcommand{\lrt}{\longrightarrow}
\newcommand{\lf}{\leftarrow}
\newcommand{\llf}{\longleftarrow}

\newcommand{\st}{\stackrel}

\newcommand{\Md}{\mathsf{Mod}}

\def\ker{\operatorname{\mathsf{ker}}}

\newcommand{\T}{\mathcal{T}}

\newcommand{\D}{\mathbb{D} }
\newcommand{\Q}{\mathcal{Q} }

\newcommand{\C}{\mathscr{C} }

\newcommand{\G}{\mathcal{G} }

\newcommand{\m}{\mathfrak{m}}

\newcommand{\p}{\mathfrak{p} }

\def\proj{\operatorname{\mathsf{proj}}}
\def\inj{\operatorname{\mathsf{inj}}}

\newcommand{\id} {\mathsf{id}}

\def\Ext{\operatorname{{\mathsf{Ext}}}}

\def\hom{\operatorname{{\mathsf{Hom}}}}
\def\uhom{\operatorname{\underline{\mathsf{Hom}}}}

\def\Tor{\operatorname{\mathsf{Tor}}}
\def\hom{\operatorname{\mathsf{Hom}}}

\def\G{\mathcal{G}}

\def\syz{\mathsf{\Omega}}
\def\coh{\operatorname{\mathsf{coh}}}

\def\x{\mathsf {X}}

\def\b{\mathsf {b}}

\def\p{\mathcal P}
\def\A{\mathcal A}

\def\U{\mathcal U}

\def\nat{\operatorname{\mathsf{Nat}}}

\def\ruf{\operatorname{\mathsf{RUF}}}
\def\luf{\operatorname{\mathsf{LUF}}}
\def\c{\operatorname{\underline{\mathscr{C}}}}
\def\Syz{\operatorname{\boldsymbol{\mathsf\Omega}}}

\def\al{\operatorname{\boldsymbol{\alpha}}}
\def\be{\operatorname{\boldsymbol{\beta}}}
\def\ga{\operatorname{\boldsymbol{\gamma}}}
\def\th{\operatorname{\boldsymbol{\theta}}}

\def\et{\operatorname{\boldsymbol{\eta}}}

\newtheorem{theorem}{Theorem}[section]

\newtheorem{lem}[theorem]{Lemma}
\newtheorem{prop}[theorem]{Proposition}

\theoremstyle{definition}
\newtheorem{dfn}[theorem]{Definition}
\newtheorem{example}[theorem]{Example}

\newtheorem{obs}[theorem]{Observation}
\newtheorem{rem}[theorem]{Remark}
\newtheorem{s}[theorem]{}

\theoremstyle{plain}

\theoremstyle{definition}

\numberwithin{equation}{section}

\begin{document}

\title[Triangulated structure of phantom stable categories]
{phantom stable categories of $n$-Frobenius categories are triangulated}
\author[Bahlekeh, Fotouhi, Salarian and Sartipzadeh]{Abdolnaser Bahlekeh, Fahimeh Sadat Fotouhi, Shokrollah Salarian and Atousa Sartipzadeh}

\address{Department of Mathematics, Gonbad-Kavous University, Postal Code:4971799151, Gonbad-Kavous, Iran}
\email{bahlekeh@gonbad.ac.ir}

\address{ School of Mathematics, Institute for Research in Fundamental Science (IPM), P.O.Box: 19395-5746, Tehran, Iran}\email{ffotouhi@ipm.ir}

\address{Department of Pure Mathematics, Faculty of Mathematics and Statistics,
University of Isfahan, P.O.Box: 81746-73441, Isfahan,
 Iran and \\ School of Mathematics, Institute for Research in Fundamental Science (IPM), P.O.Box: 19395-5746, Tehran, Iran}
 \email{Salarian@ipm.ir}

\address{Department of Pure Mathematics, Faculty of Mathematics and Statistics,
University of Isfahan, P.O.Box: 81746-73441, Isfahan, Iran}
 \email{asartipz@gmail.com}


\thanks{This work is based upon research funded by Iran National Science Foundation (INSF) under project No. 4048273.}
\thanks{The research of the second author was in part supported by a grant from IPM  (No. 1405130044).}

\begin{abstract}
Let $n$ be a non-negative integer. Motivated by the universal property of the stable category of Frobenius categories, the authors in \cite{bfss} generalized the stabilization of Frobenius categories to $n$-Frobenius categories, defining the phantom stable category. For  an $n$-Frobenius category $\C$, this consists of 
 a pair $(\C_{\p}, T)$, where $\C_{\p}$ is an additive category having the same objects as $\C$ and $T:\C\rt\C_{\p}$  an additive covariant functor that vanishes on $n$-$\Ext$-phantom morphisms and sends $n$-$\Ext$-invertible morphisms to isomorphisms, and $T$ has the universal property with respect to these conditions. The existence of the phantom stable category $(\C_{\p}, T)$ and its several interesting properties have appeared in \cite{bfss}.    In this paper, we  show  that the syzygy functor $\syz$, constructed from $n$-projectives, from $\C$ to $\C_{\p}$ is not only an additive functor, but also  it induces an auto-equivalence functor $\Syz$  on $\C_{\p}$. Then, as the main result, it is proved that phantom stable category   $(\C_{\p}, T)$ is triangulated, with $\Syz$ serving as its shift functor. 
\end{abstract}
\maketitle

\tableofcontents

\section{Introduction}
The significance of Frobenius categories lies in their profound connection to triangulated categories. A fundamental result, due to Happel \cite{ha1}, establishes that the stable category of a Frobenius category carries a natural triangulated structure, and in particular, most of the triangulated categories arise in this way, see for example \cite{bu, ha1}. Following  Keller \cite{ke}, such categories are termed “algebraic triangulated” categories and are distinguished by their admitting a natural  dg-enhancement in the sense of Bondal and Kapranov \cite{bk}. This makes them a more tractable and conceptually rich class than general triangulated categories.

 {These facts encouraged  us to study the notion of $n$-Frobenius categories and make an attempt to extend the concept of the stabilization of Frobenius categories to  $n$-Frobenius categories. Our principal focus in this direction is to consider the universal property for the stable category of a Frobenius category. In order to explain our idea more explicitly, we need to recall  some notions.  Assume that $\A$ is an abelian category and $\C$ is a full additive extension-closed subcategory of $\A$. The exact category $\C$ is said to be an $n$-Frobenius category, with $n$ a non-negative integer, if  $\C$ has enough $n$-projective and $n$-injective objects and these two classes of objects coincide, where $n$-projectivity and $n$-injectivity are defined via the vanishing of $(n+1)$-th Yoneda $\Ext$. Thus  Frobenius categories are indeed  0-Frobenius categories in our sense. It is known that  any abelian category with non-zero $n$-projective objects admits a non-trivial $n$-Frobenius subcategory, see \cite[Theorem 2.15]{bfss}.

Assume that $\C$ is an $n$-Frobenius category. A given morphism $f: M\rt N$ in $\C$, is said to be an $n$-$\Ext$-phantom (resp. invertible) morphism, provided that the induced natural transformation $\Ext^{n+1}(-, f):\Ext_{\C}^{n+1}(-, M)\lrt\Ext^{n+1}_{\C}(-, N)$ vanishes (resp. is an equivalence of functors). Evidently, the class of all $n$-$\Ext$-phantom morphisms is an ideal of $\C$. It is established that if $\C$ is a Frobenius category, then $0$-$\Ext$-phantom morphisms are precisely those morphisms that factor through projective objects, see Proposition \ref{stable}. Considering this fact together \cite[Theorem 1]{hz}, we may revisit the stabilization of  Frobenius categories, as follows.  
\begin{dfn} Let $\C$ be a Frobenius category. An additive category $\C'$ equipped with a covariant  additive functor $T:\C\rt\C'$ is said to be the stable category of $\C$, if:\\
(1) for any $0$-$\Ext$-phantom  morphism $f$ in $\C$, $T(f)=0$ in $\C'$;\\ (2) for any additive covariant functor $T':\C\rt\D$  that satisfies the first assertion, there exists a unique additive covariant functor $F:\C'\rt\D$ such that $FT=T'$.
\end{dfn}

By the idea of the above universal property, the authors in \cite{bfss} have been able to extend the stabilization of Frobenius categories to  $n$-Frobenius categories, which is called {\em the phantom stable category}.  Precisely, we have the following.
 }
      
 \begin{dfn}(\cite[Definition 9.2]{bfss}) Assume that $\C$ is an $n$-Frobenius category. A couple $(\C_{\p}, T)$, with $\C_{\p}$ an additive category and $T:\C\lrt\C_{\p}$  an additive covariant functor, is said to be the {\it phantom stable category of $\C$}, if:\\ (1) for any $n$-$\Ext$-phantom morphism $\varphi$, $T(\varphi)=0$ in $\C_{\p}$;\\ (2)  $T(s)$ is an isomorphism in $\C_{\p}$, for any $n$-$\Ext$-invertible morphism $s$;\\ (3) any additive covariant functor $T':\C\lrt\D$ satisfying the conditions (1) and (2), factors in a unique way through $T$.
\end{dfn}

The existence of the phantom stable category $(\C_{\p}, T)$ of an $n$-Frobenius category $\C$, has appeared in \cite[Theorem 9.5]{bfss}. Indeed, $\C_{\p}$ is an additive category with the same object of $\C$, and for every two objects $M, N\in\C_{\p}$, the hom-set  $\hom_{\C_{\p}}(M, N)=(\Ext^n_{\C}(M, \syz^nN)/{\p}, \delta_N)$ modulo an equivalence relation, where  $\delta_N$ is a unit conflation of $N$ of length $n$. Moreover, for each two objects $M, N\in\C$, 
 $\p(M, N)$ is the additive subgroup of $\Ext^n_{\C}(M, N)$ consisting of all conflations  $\ga$ of length $n$,  arising from a pull-back along a morphism $M\st{f}\rt P$ with $P$ $n$-projective. Furthermore, the addition of morphisms is defined via the Baer sum operation, while the composition is defined in terms of pull-back (or equivalently, push-out) diagrams.  In addition, $T:\C\rt\C_{\p}$ is an additive functor which is identical on objects and for a morphism $f:M\rt N$ in $\C$, $T(f)=(\overline{\delta_Nf}, \delta_N)$  (see Section 2  for the notations and the detailed construction of $\C_{\p}$).
  
 Since the stable category of a Frobenius category is triangulated, where the shift functor is given by the syzygy functor, a natural question arises:  does the phantom stable category $(\C_{\p}, T)$ of an $n$-Frobenius category $\C$ admit a triangulated structure? The main goal of this paper is to answer this question  affirmatively. Precisely, 
 we show that  the syzygy functor $\syz$, defined using $n$-projective objects from $\C$ to $\C_{\p}$, provides an additive covariant functor that sends each $n$-$\Ext$-phantom (resp. invertible) morphism to zero (an isomorphism) in $\C_{\p}$, see Theorem \ref{fun}. This, together with the universal property of the phantom stable category,  induces an auto-equivalence functor $\Syz$ on $\C_{\p}$, which is  again called the syzygy funcor, see Theorems \ref{fun1} and \ref{sem}. As the main result of the paper, we prove that  the phantom stable category $(\C_{\p}, T)$ is a triangulated category, where the shift functor is given by $\Syz$. Moreover, the exact triangles are given up to isomorphism of triangles by  $\Syz N\st{T(g)}\lrt C\st{T(h)}\lrt M\st{T(f)}\lrt N$, that arises from the pull-back of a syzygy sequence $\syz N\rat Q\re N$ with $Q\in n$-$\proj\C$, along a morphism $f:M\rt N$ in $\C$, see Theorem \ref{lt}.  At the end of the paper, we show that if $\C$ is a Frobenius category, then its stabilization as a 0-Frobenius category, $(\c, F)$, is triangle equivalent to the  stabilization $(\C_{\p}, T)$ of $\C$, viewed as an $n$-Frobeius category,  where $n\geq 1$ is an integer,  see Theorem \ref{thm}.



The paper is organized as follows. In Section 2, we recall the definition of $n$-Frobenius categories, provide several examples, and  briefly outline the construction  of the phantom stable category from \cite{bfss}. Section 3 is devoted to defining the shift functor of the triangulated structure. Indeed,  we show that  the syzygy {functor} $\syz$, using $n$-projective objects, from $\C$ to $\C_{\p}$, induces an  auto-equivalence functor $\Syz$ on $\C_{\p}$. To verify that the  phantom stable category is triangulated, we need to describe commutative diagrams in $\C_{\p}$ in terms of commutative diagrams in $\C$. This is carries out  in Section 4. Finally, in Section 5,  we prove that the phantom stable category $(\C_{\p}, T)$ admits a triangulated structure.

Throughout the paper, $\A$ is an abelian category and $\C$ a full additive extension-closed 
 subcategory of $\A$. It is known that $\C$ is an exact category, where conflations are precisely short exact sequences in $\A$ whose terms are in $\C$, see \cite[Lemma 10.20]{buh}. The integer $n\geq 0$ is fixed.


\section{Phantom stable categories }
This section is devoted to giving a brief description of the structure of the phantom stable category of an $n$-Frobenius category.  First, we recall the definition of  an $n$-Frobenius category and  some of  its properties  from \cite{bfss}. Also, several examples of $n$-Frobenius categories are presented.

\begin{s}Assume that $t\geq 1$ is an integer. An extension  $0\rt N\rt X_{t-1}\rt\cdots\rt X_0\rt M\rt 0$ in $\A$, is said to be a  conflation (of length $t$) in $\C$, if it is obtained by splicing $t$ conflations  in $\C$, and it will be denoted by $N\rat X_{t-1}\rt\cdots\rt X_{0}\re M$. So conflations are exactly conflations of length 1.  For a given conflation  $N\st{f}\rat Z\st{g}\re M$ in $\C$,  $f$ (resp. $g$) is called an inflation (resp. a deflation), see \cite{ke}. The set of all equivalence classes of conflations (of length $t$) of the form $N\rat X_{t-1}\rt\cdots\rt X_{0}\re M$, is denoted by $\Ext^t_{\C}(M, N)$, see for detail  \cite[Chapter VII]{mit}. We also set $\Ext^0_{\C}(-, -)=\hom_{\C}(-, -)$.  For the sake of simplicity, in $\Ext^t_{\C}(-, -)$,  the subscript $\C$ will be suppressed. Furthermore, if there is no ambguity, we often write `a conflation' instead of `a conflation of length $t$'.
\end{s}

\begin{dfn} Assume that $n$ is a non-negative integer.  \\
(1) A given object $P\in\C$ (resp. $I\in\C$) is said to be an {\it $n$-projective} (resp. {\it $n$-injective}) object of $\C$, if $\Ext^i(P, X)=0$ (resp. $\Ext^i(X, I)=0$) for all integers $i>n$ and all objects $X\in\C$. The class of all $n$-projective (resp. $n$-injective) objects will be denoted by $n$-$\proj\C$ (resp. $n$-$\inj\C$). \\
(2) The exact category $\C$ is said to have enough $n$-projectives, provided that each object $M$ in $\C$ fits into a deflation $P\re M$ where $P$ is $n$-projective. Dually one has the notion of having enough $n$-injectives.\\
(3) The exact category $\C$ is called $n$-Frobenius, if
it has enough $n$-projectives and $n$-injectives and $n$-$\proj\C$ coincides with $n$-$\inj\C$.\\ It follows from the definition that  0-Frobenius categories are indeed Frobenius categories in the usual sense. Moreover, any $n$-Frobenius category is a $k$-Frobenius category, for all $k> n$.
\end{dfn}

\begin{rem}
{{It should be remarked that there are also other notions related to the  Frobenius categories and their stable categories in  \cite{ja} and \cite{lz}, where the notions of Frobenius $n$-exact categories and Frobenius $n$-exangulated categories, for each positive integer $n$, have been defined and studied.  We point out that, if the base category is exact and $n=1$, then these notions recover the classical notion of a Frobenius exact category. So, in this case, these concepts coincide with the 0-Frobenius category in our sense. }

We should emphasize that although the aforementioned three concepts are related to the classical notion of Frobenius exact categories, their ideas are totally different.
{Indeed, $n$-exangulated categories defined in \cite{hln, hln1}, as a higher dimensional analog of extriangulated categories which were introduced by Nakaoka and Palu \cite{np} by extracting those properties of $\Ext^1$ on exact categories and on triangulated categories that seem relevant from the perspective of cotorsion pairs. However, our motivation for studying $n$-Frobenius categories and their stabilizations lies in the fact that there exist categories that have enough $n$-projectives with no non-zero projective objects.} For example, there are no projective objects in the category of quasi-coherent sheaves over the projective line $\mathbf{P^1}(R)$, where $R$ is a commutative ring with identity, see \cite[Corollary 2.3]{ee} and also \cite[Exercise III 6.2(b)]{har}. However, a result of Serre indicates that the category of coherent sheaves over a projective scheme has enough locally free sheaves, see \cite[Corollary 5.18]{har}. More generally, as mentioned in \cite{bfss}, the argument given in the proof of \cite[Lemma 1.12]{or} reveals that for a semi-separated noetherian scheme $\x$ of finite Krull dimension, there exists a non-negative integer $n$ such that locally free sheaves of finite rank, are $n$-projective objects in the category of coherent sheaves, $\coh(\x)$, and in particular, the subcategory of $n$-projective objects of $\coh(X)$, is non-trivial.
}
\end{rem}

In what follows, we provide some examples of $n$-Frobenius categories.

\begin{example}\label{remexam} Let $(R, \m)$ be a $d$-dimensional commutative noetherian local ring.\\ (1) The category of all (finitely generated) $R$-modules of finite Gorenstein projective dimension is a $d$-Frobenius category, and moreover, its $d$-projective objects are all (finitely generated) modules with finite projective dimension. Particularly, if $R$ is Gorenstein, then the category of all (finitely generated) $R$-modules, is a $d$-Frobenius category, see \cite[Example 2.16]{bfss}.\\ (2) The class of all Gorenstein flat $R$-modules, $\mathsf{GF}(R)$, is a $d$-Frobenius category such that its $d$-projective objects are flat modules. To see this, first one should note that since $R$ is noetherian, $\mathsf{GF}(R)$ is closed under extensions, see \cite[Theorem 3.7]{holm}. So it will be an exact subcategory of $\Md(R)$, the category of all $R$-modules. Now assume that $M\in\mathsf{GF}(R)$ and $\mathbf{F}_{\bullet}$ is its complete flat resolution. In view of \cite[Theorem 4.18]{ms}, $\hom_R(\mathbf{F}_{\bullet}, C)$ is an acyclic complex, for all cotorsion flat modules $C$, implying that $\Ext^i_R(M, C)=0$ for any $i\geq 1$. Furthermore, as $R$ is $d$-dimensional, the projective (and so, cotorsion flat) dimension of any flat module is at most $d$, see \cite[Theorem 4.2.8]{xu} and \cite[Theorem 3.4]{ho}. Consequently, $\Ext_R^i(M, F)=0=\Ext^i_R(F, M)$ for all flat modules $F$ and $i\geq d+1$. In particular, applying \cite[Proposition 2.18]{bfss} yields that $d$-projective (also $d$-injective) objects of $\mathsf{GF}(R)$ are exactly flat modules, and so, the claim follows.\\ (3) Let $(R, \m)$ be a commutative noetherian Gorenstein local ring and $\Q$ a finite acyclic quiver. Then the path algebra $R\Q$ is a Gorenstein order in the sense of Auslander \cite{auslander1967functors}. Assume that $\G$ is the category consisting of all lattices over $R\Q$, that is, those $R\Q$-modules $M$ which are finitely generated Gorenstein projective over $R$.
Since the category $\mathsf{Gp}(R)$ of Gorenstein projective $R$-modules is closed under extensions, the same is true for $\G$. Thus, it will be an exact subcategory of $\Md(R\Q)$. In addition,
according to the second short exact sequence appeared in \cite[Page 4429]{ehs} and its dual, any  $R\Q$-module $M$ which is projective over $R$, is a 1-projective and 1-injective object in $\Md(R\Q)$, and particularly, $\G$ will be a 1-Frobenius category. We should emphasize that $\G$ would not be a Frobenius category,  whenever $\Q$ has at least two vertices.
\end{example}

Recall that a left module $M$ over an associtive ring $\Lambda$ is said to be Gorenstein projective (resp. Gorenstein flat), if there exists an acyclic complex of left projective  (resp. flat) modules: $\mathbf{F}_{\bullet}:\cdots\rt F_1\rt F_0\rt F_{-1}\rt\cdots$ such that $M=\ker(F_0\rt F_{-1})$ and for any left projective module $Q$ (resp.
right injective module $I$), the functor  $\hom_{\Lambda}(-, Q)$, (resp. $I\otimes_{\Lambda}-$) leaves the complex $\mathbf{F}_{\bullet}$ exact. Such a complex $\mathbf{F}_{\bullet}$ is called a complete projective (resp. flat) resolution of $M$, see \cite{enochs2011relative}.

In the remainder of the paper, unless stated otherwise, we assume that $\C$ is an $n$-Frobenius category. 

\begin{s}\label{ss}
(1) Assume that  $N$ is an object of $\C$. For any integer $k\geq 1$, $N$ fits into conflations of length $k$, $\syz^kN\rat P_{k-1}\rt\cdots\rt P_0\re N$ and $N\rat P^1\rt\cdots\rt P^k\re\syz^{-k}N$ such that $P_i, P^i$'s  are  $n$-projective, which will be called   unit conflations, and also syzygy sequences. Also $\syz^kN$ is called a $k$-th syzygy of $N$. Clearly, unit conflations are not uniquely determined. We denote the class of all  unit conflations of length $k$  ending at  $N$ (resp.  beginning with  $N$) by $\U_k(N)$ (resp. $\U^k(N)$). Unit conflations usually will be depicted by $\delta$. It is worth noting that for any object $N\in\C$, $\U_0(N)$ and $\U^0(N)$ only contain the identity morphism $1_N$.\\
(2) Assume that $\delta, \delta' \in \U^n(N)$ (resp. $\delta_1, \delta_2 \in\U_n(N)$). Then there exist deflations (resp. inflations) $a, a'$ with $n$-projective kernels (resp. cokernels) such that $\delta a=\delta'a'$ (resp. $a\delta_1=a'\delta_2$), see \cite[Proposition 5.5]{bfss}. In this situation $[\delta a, \delta'a']$  (resp. $[a\delta_1, a'\delta_2]$) is called a  co-angled pair (resp.  an angled pair). In some cases we also denote it by $\delta\st{a}\llf\delta''\st{a'}\lrt\delta'$ (resp. $\delta_1\st{a}\lrt\delta''\st{a'}\llf\delta_2$), where $\delta''=\delta a$  (resp. $\delta''=a\delta_1$).\\ (3) Suppose that $\ga\in\Ext^n(M, \syz^n N)$ is given, where $n\geq 1$ is an integer. It is known that there exist unit conflations $\delta\in\U^n(\syz^n N)$, $\delta'\in\U_n(M)$ and morphisms $f,g$ in $\C$ such that $\ga=\delta f$ and $\ga=g\delta'$, see \cite[Proposition 5.2]{bfss}. These are called  a right and a left unit factorization of $\ga$, respectively, and are abbreviated by $\ruf$ and $\luf$.\\ (4) Assume that $f:M\rt N$ is a morphism in $\C$ and $\delta_N\in\U_1(N)$. So taking an $\luf$ of $\delta_Nf$ gives us a unit conflation $\delta_M\in\U_1(M)$ and a morphism $f':\syz M\rt\syz N$ in $\C$ such that $f'\delta_M=\delta_Nf$.
\end{s}

As mentioned in the introduction, modifying the stable category of a Frobenius category, we define the stabilization of  an $n$-Frobenius category $\C$, which is called the phantom stable category. In order to give its structure, we need to recall some notions.
\begin{s}\label{pp}
{\sc $\p$-subfunctor.} For each two objects $X, Y$ in $\C$,  $\p(X, Y)$ is the additive subgroup of $\Ext^n(X, Y)$ consisting of all conflations $\ga$ that arise as a pull-back along a morphism $X\st{f}\rt P$ with $P$ $n$-projective. Equivalently, it arises from a push-out along a morphism $Q\st{f'}\rt Y$ with $Q$ $n$-projective, see \cite[Proposition 4.3]{bfss}. Conflations of these types are referred to as $\p$-conflations. If there is no ambiguity, we denote $\p(-, -)$ by $\p$. One should note that as being a $\p$-conflation is preserved under pull-back and push-out,  $\p$ will be a subfunctor of $\Ext^n$, see \cite{as, fght}. It should be pointed out that, since $\Ext^0(-, -)=\hom_{\C}(-, -)$, over Frobenius categories (or equivalently, 0-Frobenius categories), the  push-out as well as the pull-back of a conflation along a morphism can be interpreted as composition of morphisms in $\C$. Particularly, in this situation,
$\p$-conflations are exactly those morphisms that factor through a projective object.\\
Assume that $M\st{f}\rt N$ is a morphism in $\C$. Then for any object $X\in\C$, there exist induced maps $\Ext^n(X, M)/{\p}\rt\Ext^n(X, N)/{\p}$ and $\Ext^n(N, X)/{\p}\rt\Ext^n(M, X)/{\p}$ that assign  each object $\ga+\p$ to ${f\ga}+\p$ and ${\ga f}+\p$, respectively. In what follows, a given object $\ga+\p\in\Ext^n/{\p}$, is denoted by $\bar{\ga}$. Here $f\ga$ (resp. $\ga f$) stands for the push-out (resp. pull-back) of $\ga$ along $f$.
\end{s}
Here we recall two kinds of morphisms that play an essential role in the stabilization of $n$-Frobenius categories.  

\begin{s}{\sc $n$-$\Ext$-phantom and $n$-$\Ext$-invertible morphisms.} Assume that $f:M\rt N$ is a morphism in $\C$.\\ (1)
$f$ is called an $n$-$\Ext$-invertible morphism, provided that   for any object $X\in\C$, the induced morphism $\Ext^n(X, M)/{\p}\rt\Ext^n(X, N)/{\p}$ is an isomorphism. It can be seen that,  this is equivalent to stating that $\Ext^n(N, X)/{\p}\rt\Ext^n(M, X)/{\p}$) is also an isomorphism,  see \cite[Corollaries 3.4 and 4.8]{bfss}. These results imply that $f$ is an $n$-$\Ext$-invertible morphism if and only if the natural transformation $\Ext^{n+1}(-, f):\Ext^{n+1}(-,M)\lrt\Ext^{n+1}(-,N)$ (or equivalently, $\Ext^{n+1}(f, -)$) is an equivalence of functors. 
It should be pointed out that $n$-$\Ext$-invertible morphisms are called quasi-invertible morphisms in \cite{bfss}. Such morphisms are also called quasi-isomorphisms in \cite{hr}. It is known that morphisms with kernel and cokernel in $n$-$\proj\C$, are $n$-$\Ext$-invertible morphisms, see \cite[Corollary 3.6]{bfss}.\\
 (2) $f$ is said to be an $n$-$\Ext$-phantom morphism, if   for any object $X\in\C$, the induced morphism $\Ext^n(X, M)/{\p}\rt\Ext^n(X, N)/{\p}$  (or equivalently, $\Ext^n(N, X)/{\p}\lrt\Ext^n(M, X)/{\p}$) is zero, see \cite[Theorem 6.7 and Definition 6.8]{bfss}. By \cite[Corollary 6.9]{bfss},  $f$ is an $n$-$\Ext$-phantom morphism, if and only if  the natural transformation $\Ext^{n+1}(-, f):\Ext^{n+1}(-,M)\lrt\Ext^{n+1}(-,N)$ (or equivalently, $\Ext^{n+1}(f, -)$) vanishes.  It is evident that a morphism factoring through an $n$-projective object in $\C$ is an $n$-$\Ext$-phantom morphism. Moreover, \cite[Proposition 6.2]{bfss} and its dual, reveal that $f$ is $n$-$\Ext$-phantom if and only if for any $\delta\in\U_n(N)$ (or equivalently, $\delta'\in\U^n(M)$), $\delta f$ ($f\delta'$) is a $\p$-conflation.
\end{s}

\begin{rem}It is worth noting that  1-$\Ext$-phantom morphisms are $\p$-phantom morphisms in the sense of Fu et. al. \cite{fght}. Moreover, since by \cite[Corollary 6.9]{bfss}, a given morphism $f$ is an $n$-$\Ext$-phantom morphism if and only if $\Ext^{n+1}f=0$, $n$-$\Ext$-phantom morphisms are indeed $(n+1)$-$\Ext$-phantom morphisms in the sense of Mao \cite{mao}.
\end{rem}

\begin{s}
The concept of phantom morphisms arises in topology, in the study of maps between CW-complexes \cite{mc}. A map $f: X \rt Y$ between CW-complexes is said to be a phantom map, if its restriction to each skeleton $X_n$ is null homotopic.
This notion was extended later by Neeman \cite{ne} into the setting of a triangulated category, and also developed in the stable category of a finite group ring by Benson and Gnacadja \cite{gn, be2, be1, be}. The notion of a phantom morphism has been generalized to the category of $R$-modules over an associative ring $R$ by Herzog in \cite{he}, where he calls a morphism $f: M\rt N$ of left $R$-modules a phantom morphism if the induced natural transformation $\Tor^R_1(-, f):\Tor^R_1(-, M)\rt\Tor^R_1(-, N)$ of homology functors vanishes. Similarly, the morphism $f$ is said to be an $\Ext$-phantom morphism \cite{hext}, if the induced natural transformation $\Ext^1_R(-, f):\Ext^1_R(-,M)\rt\Ext^1_R(-,N)$ of cohomology functors is zero, over finitely presented left $R$-modules. As a higher dimensional generalizations of phantom morphisms and $\Ext$-phantom morphisms, Mao \cite{mao, mao1}  introduced and studied the concepts of $n$-phantom morphism and $n$-$\Ext$-phantom morphisms, where $n\geq 1$. Assume that $R$ is a left and right noetherian ring and $f:M\rt N$ is a morphism  in $\mathsf{Gp}(R)$, the category of all finitely generated Gorenstein projective right $R$-modules. By \cite[Theorem 1.2]{bs}, $\Tor_1^R(-, f)=0$ over $\mathsf{Gp}(R^{op})$ if and only if $\Ext^1_R(-, f)=0$ over $\mathsf{Gp}(R)$. So, over the category of Gorenstein projective modules, the notions of phantom morphisms and $\Ext$-phantom morphisms coincide. 
\end{s}

\begin{rem}\label{d}Assume that $R$ is a commutative Gorenstein local ring with $\mathsf{dim} R=d$. As mentioned in Example \ref{remexam}(1),  $\mathsf{mod} R$ is a $d$-Frobenius category. Suppose $f:M\rt N$ is a morphism in $\mathsf{mod}R$. By \cite[Proposition 2.11]{bs}, the natural transformation $\Ext^{d+1}_R(-, f)$ vanishes over $\mathsf{mod}R$ if and only if the same is true for $\Tor_{d+1}^R(-, f)$. So, $f$ is a $(d+1)$-phantom morphism if and only if it is a $(d+1)$-$\Ext$-phantom morphism in the sense of Mao \cite{mao}. In particular, in this case, $f$ is a $d$-$\Ext$-phantom morphism in our sense.
\end{rem}

The result below specifies $n$-$\Ext$-phantom and also $n$-$\Ext$-invertible morphisms in case $n=0$.
\begin{prop}\label{stable}Let $\C$ be a Frobenius category and $f:M\rt N$ a morphism in $\C$. Then the following assertions hold.
\begin{enumerate}\item  $f$ is { a  {\rm 0}-$\Ext$-phantom morphism if and only if} it factors through a projective object in $\C$. \item { $f$ is a  {\em 0}-$\Ext$-invertible  morphism if and only if} there are $P,Q\in\proj\C$ such that $M\oplus Q\st{{{\tiny {\left[\begin{array}{ll} f & g_1 \\ l & g_2 \end{array} \right]}}}}\lrt N\oplus P$ is an isomorphism in $\C$.
\end{enumerate}
\end{prop}
\begin{proof}Since the `if' parts are obvious, we shall prove just the `only if' parts.\\
(1) As $f:M\rt N$ is a 0-$\Ext$-phantom morphism, the natural tranformation $\Ext^1(f,-):\Ext^1(N, -)\lrt\Ext^1(M, -)$ vanishes. In particular, for a given syzygy sequence $\delta_N\in\Ext^1(N, \syz N)$, $\delta_Nf$ will be a split exact sequence, and then, $f$ factors through a projective object.\\ (2) Since $\C$ is a Frobenius category, one may take an inflation $M\st{l}\rt P$, where $P\in\proj\C$. {Also, as $f$ is a 0-$\Ext$-invertible morphism, by \cite[Lemma 3.3]{bfss}, one obtains a conflation $M\st{[f~~l]^t}\rt N\oplus P\rt Q$ in $\C$, where $Q\in\proj\C$. Thus this sequence is split, and then, we get the isomorphism $M\oplus Q\st{{{\tiny {\left[\begin{array}{ll} f & g_1 \\ l & g_2 \end{array} \right]}}}}\lrt N\oplus P$, as desired.}
\end{proof}

In the following, we remind the definition of the stable category of a Frobenius category.
\begin{s}\label{0}
Assume that $\C$ is a Frobenius category. The stable category $\c$ of $\C$ is defined as follows:
the objects of $\c$ are the same as those of $\C$, and the hom-set $\hom_{\c}(M, N)$ is the quotient $\underline{\hom}_{\C}(M, N)$ of the abelian group $\hom_{\C}(M, N)$ by the subgroup consisting of all morphisms $M\rt N$ factoring through projective objects. It is known that there is an additive quotient functor $F:\C\lrt\c$ which is the identity on objects and maps each morphism $f$ to its equivalence class $\bar{f}$.
\end{s}
In view of Proposition \ref{stable}(1), 0-$\Ext$-phantom morphisms are those, factoring through projective objects. This, combined with \cite[Theorem 1]{hz} yields the next result, which says that the functor $F$ admits a universal property.
\begin{theorem}\label{sta} With the notation above, consider the couple $(\c, F)$. The following assertions hold.\\ (1) For any $0$-$\Ext$-phantom  morphism $f$ in $\C$, $F(f)=0$ in $\c$.\\ (2) For any additive covariant functor $F':\C\rt\D$ that satisfies the condition (1), there is a unique additive functor $T:\c\rt\D$ making the following diagram commute
{\footnotesize\[\xymatrix{\C\ar[r]^{F}\ar[dr]_{F'}~& \c\ar[d]^{T}& \\ & \D.& }\]}
\end{theorem}

As mentioned in the introduction,  the universal property of the stable category of a Frobenius category, when regarded as a 0-Frobenius category, prompts the authors  to define the stabilization of $n$-Frobenius categories, as follows.

 \begin{dfn}\label{dd}(\cite[Definition 9.2]{bfss})  By the phantom stable category of $\C$, we mean an additive category $\C_{\p}$, equipped with an  additive covariant functor  $T:\C\lrt\C_{\p}$   such that \\ (1) for any $n$-$\Ext$-phantom morphism $\varphi$, $T(\varphi)=0$ in $\C_{\p}$;\\ (2)  $T(s)$ is an isomorphism in $\C_{\p}$, for any $n$-$\Ext$-invertible morphism $s$;\\ (3) any additive covariant functor $T':\C\lrt\D$ satisfying the conditions (1) and (2), factors in a unique way through $T$.
\end{dfn}

\begin{rem}\label{0frob}
Assume that $\C$ is a Frobenius category. Then  $(\underline{\C}, F)$ is the phantom stable category of $\C$. To see this, considering Theorem \ref{sta}, we only need to show that the second condition of Definition \ref{dd}, automatically holds for this case. So assume that  $f:M\rt N$ is  a $0$-$\Ext$-invertible morphism in $\C$. By  Proposition \ref{stable}(2), there is an isomorphism $f':M\oplus Q\st{{{\tiny {\left[\begin{array}{ll} f & g_1 \\ l & g_2 \end{array} \right]}}}}\lrt N\oplus P$ in $\C$, where $P, Q\in\proj\C$, implying that $F(f)=F(f')$. This, in particular, yields that $F(f)$ is an isomorphism in $\underline{\C}$, as claimed.
\end{rem}

\begin{s}Let $a:X\rt X'$ be an $n$-$\Ext$-invertible morphism and $\ga\in\Ext^n(X, Y)$. In view of \cite[Corollary 4.10]{bfss}, there exists $\ga'\in\Ext^n(X', Y)$ such that $\ga-\ga'a$ is a $\p$-conflation. In this case, for simplicity, we write $\ga'=_{_{\p}}\ga a^{-1}$. Also, {for a given object} $\be\in\Ext^n(Y, X')$, there exists $\be'\in\Ext^n(Y, X)$ such that $\be -a\be'$ is a $\p$-conflation. Then we write $\be'=_{_{\p}}a^{-1}\be$.
\end{s}

 It should be highlighted that the existence of the phantom stable category $(\C_{\p}, T)$ of an $n$-Frobenius category  $\C$, which we usually suppress $T$, has been proved in \cite[Theorem 9.5]{bfss}. Here, we only outline the structure of this category.

\begin{s}{\sc Structure of $(\C_{\p}, T)$.}\label{co} The phantom stable category $(\C_{\p}, T)$ of $\C$, is constructed as follows:
\begin{itemize}\item The objects of $\C_{\p}$ are the same as those of $\C$.
\item  For any pair of objects $M, N\in\C_{\p}$, fix a unit conflation $\delta_N\in\Ext^n(N, \syz^n N)$ and set $\hom_{\C_{\p}}(M, N):=(\Ext^n(M, \syz^n N)/{\p}, \delta_N)$ modulo the relation $``\thicksim"$.\\
 For given two objects $(\bar{\ga}, \delta_N)$ and $(\bar{\ga'}, \delta'_N)$ in $\hom_{\C_{\p}}(M, N)$, we write $(\bar{\ga}, \delta_N)\thicksim(\bar{\ga'}, \delta'_N)$, if there exists an angled pair $[a\delta_N, b\delta'_N]$ such that
$a\ga-b\ga'$ is a $\p$-conflation. According to \cite[Theorem 8.2]{bfss}, $``\thicksim"$ is an equivalence relation on $\bigcup_{ \delta_N\in\U_n(N)}(\Ext^n(M, \syz^n N)/{\p}, \delta_N)$.
\item
 Composition of morphisms is defined as follows:\\
For two morphisms $(\bar{\ga}, \delta_N)\in\hom_{\C_{\p}}(M, N)$ and $(\bar{\be}, \delta_K)\in\hom_{\C_{\p}}(N, K)$, take an $\ruf$ $\ga=\delta f$ of $\ga$ and set $(\bar{\be}, \delta_K)\circ(\bar{\ga}, \delta_N):=(\overline{((\be a)b^{-1})f}, \delta_K)$, where $\delta_N\st{a}\llf\delta_1\st{b}\lrt\delta$ is a co-angled pair. Well-definedness of this definition comes from
\cite[Lemma 7.3]{bfss} and \cite[Theorem 7.4]{bfss}. Furthermore, the equivalence relation is compatible with the composition, thanks to  \cite[Proposition 8.6]{bfss}. 
\item  $T:\C\rt\C_{\p}$ is an additive covariant functor that is identical on objects and for a given morphism $f: M\rt N$ in $\C$,  $T(f)$ is defined by first fixing a unit conflation $\delta_N\in\Ext^n(N, \syz^nN)$ and setting $T(f)=(\overline{\delta_Nf}, \delta_N)$.
\end{itemize}
It is easily seen that, for  any object $M\in\C_{\p}$, $1_M= (\bar{\delta}_M, \delta_M)$.
In addition,  \cite[Proposition 7.9]{bfss} ensures that  $``\circ"$ is associative, and also it is bilinear. Thus, $\C_{\p}$ will be an additive category.
\end{s}
Let us give explicit definition of the addition in  $\C_{\p}$. Assume that $(\bar{\ga}, \delta_N)$ and $(\bar{\ga'}, \delta'_N)$ are two objects of $\hom_{\C_{\p}}(M, N)$. Now define the addition as follows:
$$(\bar{\ga}, \delta_N)+(\bar{\ga'}, \delta'_N):=(\overline{a\ga+b\ga'}, \delta''_N),$$ where  $\delta_N\st{a}\lrt\delta''_N\st{b}\llf\delta'_N$ is an angled pair and $a\ga+b\ga'$ stands for the usual Baer sum.

\begin{prop}The addition $``+"$ is well-defined.
\end{prop}
\begin{proof}We need to show that the definition of  $``+"$ is independent of the choice of angled pairs and the equivalence relation $``\thicksim"$ is compatible with $``+"$. In order to verify the first statement, assume that $\delta_N\st{a_1}\lrt\delta_1\st{b_1}\llf\delta'_N$ is the angled pair appeared in \cite[Proposition 5.5(2)]{bfss}. Since  $\delta_N\st{a}\lrt\delta''_N\st{b}\llf\delta'_N$ is an angled pair, the equalities $\delta''_N=a\delta_N=b\delta'_N$ hold. This, combined with \cite[Proposition 5.8]{bfss} yields that $\delta''_N=h\delta_1$, for some $n$-$\Ext$-invertible morphism $h$ in $\C$. A dual  argument given in the proof of \cite[Proposition 5.7]{bfss} and get  the commutative diagram
{\footnotesize\[\xymatrix{\syz^nN\ar[r]^{a_1}\ar[dr]_{a}~& \syz^n_1N\ar[d]_{h}& {\syz'}^nN\ar[l]_{b_1}\ar[dl]_{b}\\ & {\syz''}^nN& }\]}where $\syz^nN, {\syz'}^nN, {\syz''}^nN$ and $\syz^n_1N$ denote the left end terms of $\delta_N, \delta'_N, \delta''_N$ and $\delta_1$, respectively. This leads to the equalities ${a\ga+\b\ga'}={ha_1\ga+hb_1\ga'}={h(a_1\ga+b_1\ga')}$. Now since $h$ is an $n$-$\Ext$-invertible morphism, we infer that $({\overline{a\ga+b\ga'}}, \delta''_N)=(\overline{a_1\ga+b_1\ga'}, \delta_1)$ that ensures the validity of the first assertion. To deal with the second statement, take the  morphisms $(\bar{\ga}, \delta_N), (\bar{\be}, \delta'_N)$ and $(\bar{\be'}, \delta''_N)$ in $\hom_{\C_{\p}}(M, N)$ such that $(\bar{\be}, \delta'_N)=(\bar{\be'}, \delta''_N)$. It must be proved that  $(\bar{\ga}, \delta_N)+ (\bar{\be}, \delta'_N)= (\bar{\ga}, \delta_N)+ (\bar{\be'}, \delta''_N)$. By the hypothesis, $\overline{a\be}=\overline{b\be'}$, where $\delta'_N\st{a}\lrt\delta_1\st{b}\llf\delta''_N$ is an angled pair. So considering an angled pair $\delta_N\st{a_1}\lrt\delta_2\st{b_1}\llf\delta_1$ gives angled pairs $[a_1\delta_N, b_1a\delta'_N]$ and $[a_1\delta_N, b_1b\delta''_N]$. As $\overline{a\be}=\overline{b\be'}$, we have $\overline{b_1a\be}=\overline{b_1b\be'}$, and so, $\overline{a_1\ga+b_1a\be}=\overline{a_1\ga+b_1b\be'}$. Thus the proof is complete.
\end{proof}

 \begin{rem}Since $\C$ is an $n$-Frobenius category, as mentioned  before, it is a $k$-Frobenius category, for all integers  $k> n$. It is proved in \cite[Corollary 9.8]{bfss} that its phantom stable categories as  $n$-Frobenius and  $k$-Frobenius categories, are equivalent.
\end{rem}
 
We close this section with the following remark.
\begin{rem}It is worth  pointing-out that the first and third authors, in an upcoming paper \cite{bs1}, study the natural transformations between extension functors on $\C$. Indeed, they have observed that for any  pair of objects $M, N\in\C$, there exist isomorphisms of abelian groups
\[\begin{array}{lllll}
\Ext^n(M, \syz^nN)/{\p} & \cong \nat(\Ext^{n+1}(-, M), \Ext^{n+1}(-, N))\\
& \cong \nat(\Ext^{n+1}(N, -), \Ext^{n+1}(M, -)),
\end{array}\]
where  $\nat(\Ext^{n+1}(-, M), \Ext^{n+1}(-, N))$ stands for the group of all  natural transformations from $\Ext^{n+1}(-, M)$ to $\Ext^{n+1}(-, N)$. Based on this observation,  the hom-set $\hom_{\C_{\p}}(M, N)$ can be taken as $\nat(\Ext^{n+1}(-, M), \Ext^{n+1}(-, N))$, and in particular, the addition and composition of morphisms corresponds to the usual addition and composition of natural transformations. Moreover, $T:\C\lrt\C_{\p}$ becomes a covariant additive functor, which is identity on objects, and for a given morphism $f$ in $\C$, $T(f)=\Ext^{n+1}(-, f)$. 

 It is also worthwhile to emphasize that this result can be viewed as  a far-reaching generalization of the Hilton-Rees theorem, as the case $n=0$, recovers the Hilton-Rees theorem, which says that if $\C$ is a Frobenius category, then for any two objects $M, N\in\C$, the isomorphisms $\uhom_{\C}(M, N)\cong\nat(\Ext^1(N, -), \Ext^1(M, -))\cong\nat(\Ext^1(-, M), \Ext^1(-, N))$ hold true, see \cite[Theorems 7 and 8]{mart}.
\end{rem}

 \section{The syzygy functor as an auto-equivalence}
The aim of this section is to show that, similar to the stable category of a Frobenius category,  the syzygy functor is an auto-equivalence functor on $\C_{\p}$. First, we remind the case for Frobenius categories.

\begin{rem}\label{rrr}Suppose that $\C$ is a  Frobenius category. For a given object $M\in\C$, one may assign an object $\syz M$ which is arisen from a syzygy sequence $\syz M\rat P\re M$, where $P\in\proj\C$. Moreover, for a given morphism $f:M\rt N$ in $\C$, one may obtain a commutative diagram with exact rows {\footnotesize \[\xymatrix{\syz M\ar@{>->}[r]\ar[d]_{\syz f}& \ \ P\ar@{->>}[r]\ar[d] \ \ & M\ar[d]_{f}\\ \syz N\ar@{>->}[r]& \ \ Q\ar@{->>}[r]\ \ & N,}\]}where $P,Q\in\proj\C$.  It is known that $\syz:\C\rt\c$ is an additive covariant functor. By making use of the universal property of the stable category $\c$, which has appeared in Theorem \ref{sta}, we show that the functor $\syz$ induces an auto-equivalence functor on $\c$.
\end{rem}

\begin{prop}\label{induce}Keeping the notation of Remark \ref{rrr}, the functor $\syz:\C\rt\c$ induces an auto-equivalence functor on $\c$.
\end{prop}
\begin{proof}Assume that $f:M\rt N$ is a 0-$\Ext$-invertible morphism in $\C$. In view of Proposition \ref{stable}(2), there is an isomorphism  $M\oplus Q\st{{{\tiny h={\left[\begin{array}{ll} f & g_1 \\ l & g_2 \end{array} \right]}}}}\lrt N\oplus P$ in $\C$, where $P, Q\in\proj\C$. So, $\syz h$ is an isomorphism in $\c$. This implies that $\syz f$ is also an isomorphism, because $\syz h=\syz f$. Next, suppose that $f$ is a 0-$\Ext$-phantom morphism. By the first assertion of Proposition \ref{stable}, $f$ factors through a projective object of $\C$, say $P$. Now since $\syz P=0$, one deduces that $\syz f=0$. Thus, by Theorem \ref{sta}(2), there is a unique additive covariant functor $\Syz:\c\rt\c$ such that $\Syz F=\syz$. In particular, for a given object $M\in\c$, $\Syz M=\syz M$ and for a given morphism $\bar{f}\in\hom_{\c}(M, N)$, $\Syz \bar{f}={\syz f}$. It is easily seen that $\Syz$ is an auto-equivalence functor on $\c$, and so, the proof is finished.
\end{proof}

In the following, we will observe that, analogous to the stable category of a Frobenius category, the syzygy functor $\syz$, using $n$-projective objects, provides an additive functor from $\C$ to $\C_{\p}$, as follows: For given two objects $M, N\in\C$, consider syzygy sequences $\delta_M:\syz M\rat P\re M$ and $\delta_N:\syz N\rat Q\re N$, and define a map $\syz_{M, N}:\hom_{\C}(M, N)\lrt\hom_{\C_{\p}}(\syz M, \syz N)$ in the following way. At first, assume that $\hom_{\C_{\p}}(\syz M, \syz N)$ is realized by $\delta_{\syz N}:\syz^{n+1}N\rat Q_n\rt\cdots\rt Q_1\re\syz N$ and $f\in\hom_{\C}(M, N)$ is given. As stated in \ref{ss}(4), there exists $\delta'_M\in\U_1(M)$ and a morphism $f':\syz'M\rt\syz N$ in $\C$ such that $f'\delta'_M=\delta_Nf$.
Now  we define $\syz(f):=(\overline{((\delta_{\syz N}f')b_1^{-1})a_1}, \delta_{\syz N})$, where  $\delta_M\st{a_1}\rt\delta_{1}\st{b_1}\lf\delta'_M$ is an angled pair. The result below asserts that $\syz:\C\lrt\C_{\p}$ is an additive covariant functor.

\begin{theorem}\label{fun}With the notation above, $\syz:\C\lrt\C_{\p}$ is a covariant additive functor.
\end{theorem}
\begin{proof} First we show that $\syz$ is well-defined. To see this, suppose that there is $\delta''_M\in\U_1(M)$ and a morphism $f'':\syz''M\rt\syz N$ with $f''\delta''_M=\delta_Nf$. So, taking an angled pair $\delta_M\st{a_2}\rt\delta_{2}\st{b_2}\lf\delta''_M$, one may have $\syz(f)=(\overline{((\delta_{\syz N}f'')b_2^{-1})a_2}, \delta_{\syz N})$. We shall prove that $(\overline{((\delta_{\syz N}f')b_1^{-1})a_1}, \delta_{\syz N})=(\overline{((\delta_{\syz N}f'')b_2^{-1})a_2}, \delta_{\syz N})$. A dual argument given in the proof of \cite[Theorem 7.4]{bfss} and applying \cite[Proposition 5.5]{bfss}, would give us the following commutative diagram in $\C$: {\footnotesize \[\xymatrix{&\syz M\ar[dl]_{a_1}\ar[dr]^{a_2}& \\ \syz_1M\ar[r]^{a_4}~& \syz_4M& \syz_2M\ar[l]_{b_4}\\ {\syz'}M\ar[u]^{b_1}\ar[r]^{a_3}~ & \syz_3 M\ar[u]^{c_4} & {\syz''}M\ar[u]^{b_2}\ar[l]_{b_3},}\]}where $a_3,b_3, a_4, b_4, c_4$ are $n$-$\Ext$-invertible morphisms.
Now similar to \cite[Proposition 5.8]{bfss}, one may find a morphism $h:\syz_3 M\rt\syz N$ such that $ha_3=f'$ and $hb_3=f''$. Namely, we have the  commutative diagram in $\C$ {\footnotesize\[\xymatrix{{\syz'}M\ar[r]^{a_3}\ar[dr]_{f'}~& \syz_3M\ar[d]^{h}& {\syz''}M\ar[l]_{b_3}\ar[dl]^{f''}\\ & \syz N.& }\]}
Consequently, one may get equalities modulo $\p$:
\[\begin{array}{lllll}((\delta_{\syz N}f')b_1^{-1})a_1& = _{\p}(((\delta_{\syz N}h)a_3) b_1^{-1})a_1\\
&=_{\p}((((\delta_{\syz N}f'')b_3^{-1})a_3)b_1^{-1})a_1\\
&=_{\p}((((\delta_{\syz N}f'')b_3^{-1})c_4^{-1})a_4)a_1\\ & =_{\p}((((\delta_{\syz N}f'')b_3^{-1})c_4^{-1})b_4)a_2\\ &=_{\p}((\delta_{\syz N}f'')b_2^{-1})a_2,
\end{array}\] meaning that $(\overline{((\delta_{\syz N}f')b_1^{-1})a_1}, \delta_{\syz N})=(\overline{((\delta_{\syz N}f'')b_2^{-1})a_2)}, \delta_{\syz N})$, giving the claim.\\ Next we show that $\syz$ is a covariant additive functor.
It is clear from the definition that $\syz(1_M)=(\bar{\delta}_{\syz M}, \delta_{\syz M})=1_{\syz M}$. Also, if $M\st{f}\rt N\st{g}\rt K$ is a pair of composable morphisms in $\C$, then we must show that $\syz(gf)=\syz(g)\syz(f)$. In view of \ref{ss}(4), there exist unit conflations $\delta'_N\in\Ext^1(N, \syz'N)$, $\delta'_M\in\Ext^1(M, \syz'M)$ and morphisms $g':\syz'N\rt\syz K$, $f':\syz'M\rt\syz'N$ such that $\delta_Kg=g'\delta'_N$ and $\delta'_Nf=f'\delta'_M$.
By the definition,
$\syz(g)=(\overline{((\delta_{\syz K}g')b_1^{-1})a_1}, \delta_{\syz K})$ and $\syz(f)=(\overline{((\delta_{\syz N}a_1^{-1})b_1 f'){b}^{-1}a}, \delta_{\syz N})$, where $[a_1\delta_N, b_1\delta'_N]$ and $[a\delta_M, b\delta'_M]$ are angled pairs. Thus, one may derive the equalities
\[\begin{array}{lllll}
\syz(g)\circ\syz(f)&=
((\overline{((\delta_{\syz K}g')b_1^{-1})a_1)a_1^{-1}b_1f'b^{-1}a}, \delta_{\syz K})\\
&=(\overline{(\delta_{\syz K}g'f')b^{-1}a}, \delta_{\syz K}).
\end{array}\]
Well-definedness of $\syz$ would imply that the latter is equal to $\syz(gf)$, and so, $\syz$ is a covariant functor. It should be noted that since $\C$ and $\C_{\p}$ are additive categories and $\syz$ preserves finite direct sums, it will be an additive functor. So the proof is finished.
\end{proof}

\begin{theorem}\label{fun1}The functor $\syz:\C\rt\C_{\p}$ induces an additive covariant endofunctor $\Syz$ on $\C_{\p}$.
\end{theorem}
\begin{proof}
By the universal property of $(\C_{\p}, T)$, it suffices to show that if {$f:M\rt N$ is an $n$-$\Ext$-invertible morphism, then $\syz(f)$ is an isomorphism, and $\syz(f)=0$, whenever $f$ is an $n$-$\Ext$-phantom morphism. First, assume that { $f$ is an $n$-$\Ext$-invertible morphism. So, for a given unit conflation $\delta_N\in\Ext^1(N,\syz^1N)$,
\cite[Lemma 3.9]{bfss} would imply that  $\delta_Nf$ is  a unit conflation,  as well. Thus, putting $\delta'_M:=\delta_Nf$, the equality $1_{\syz N}\delta'_M=\delta_Nf$ holds true. Hence, one may consider an angled pair $\delta_M\st{a}\rt\delta''_M\st{b}\lf\delta'_M$, and get the equality $\syz(f)=(\overline{(\delta_{\syz N}b^{-1})a}, \delta_{\syz N})$. Another use of \cite[Lemma 3.9]{bfss} yields that $(\delta_{\syz N}b^{-1})a$ is a unit conflation, and so, \cite[Corollary 7.10]{bfss} forces $\syz(f)$ to be an isomorphism.\\
Next, suppose that $f$ is an $n$-$\Ext$-phantom morphism. According to \ref{ss}(4), there is a unit conflation $\delta'_M\in\U_1(M)$ and a morphism $f':\syz'M\rt\syz N$ such that $f'\delta'_M=\delta_Nf$. Since $f$ is $n$-$\Ext$-phantom, by \cite[Proposition 6.2]{bfss} $\delta_Nf$ is a $\p$-conflation. Thus \cite[Proposition 6.2]{bfss} combined with \cite[Theorem 6.7]{bfss} implies that $f'$ is an $n$-$\Ext$-phantom morphism, as well. Hence, by the definition of $\syz(f)$ and making use of \cite[Proposition 6.6]{bfss}, one may deduce that $\syz(f)=0$.} So, the universal property of the phantom stable category $(\C_{\p}, T)$, gives us a unique covariant additive functor $\Syz: \C_{\p}\lrt\C_{\p}$,  such that $\Syz T=\syz$. Thus the proof is completed.}
\end{proof}

\begin{rem}\label{22}As observed in Theorem \ref{fun1}, there exists a covariant additive functor $\Syz:\C_{\p}\lrt\C_{\p}$. { Here we illustrate how the functor $\Syz$ acts on objects and morphisms in $\C_{\p},$ explicitly. For a given object $M\in\C_{\p}$, first we consider a unit conflation $\delta_M:\syz M\rat P\re M$ in $\C$. So we have $\Syz M=\syz M$. Next take two objects $M,N\in\C_{\p}$ and consider a syzygy sequence $\syz N\rat P_0\re N$ in $\C$. Assume that $\hom_{\C_{\p}}(M, N)$ and $\hom_{\C_{\p}}(\Syz M, \Syz N)$ are realized by $\delta_{N}:\syz^{n}N\rat P_{n-1}\rt\cdots\rt P_1\rt P_0\re N$ and $\delta_{\syz N}:\syz^{n+1}N\rat P_n\rt\cdots\rt P_1\re\syz N$, respectively.} {For a given object $(\bar{\ga}, \delta_N)\in\hom_{\C_{\p}}(M, N)$, we may take an $\ruf $  $\ga=\delta_{N'}g$ of $\ga$, where  $g:M\rt N'$ is a morphism in $\C$. According to \ref{ss}(4), there is a unit conflation  $\delta'_M\in\U_1(M)$ and a morphism $g':\syz'M\rt\syz N'$ in $\C$ such that $\delta_{1}g=g'\delta'_M$, where $\delta_1=\syz N'\rat P'_0\re N'$. Thus we have $\Syz(\bar{\ga}, \delta_N)=(\overline{((\delta_{\syz N'}g'){b}^{-1})a}, \delta_{\syz N}),$ where $\delta_{M}\st{a}\rt\delta''_{M}\st{b}\lf\delta'_{M}$ is an angled pair.}
\end{rem}
The following auxiliary lemma will be used in the proof of the next theorem.
\begin{lem}\label{extphantom}Consider the following commutative diagram  {\footnotesize \[\xymatrix{\syz M\ar@{>->}[r]\ar[d]_{f'}& \ \ P\ar@{->>}[r]\ar[d] \ \ & M\ar[d]_{f}\\ \syz N\ar@{>->}[r]& \ \ Q\ar@{->>}[r]\ \ & N,}\]}with $P,Q\in n$-$\proj\C$. Then $f$ is an $n$-$\Ext$-phantom morphism if and only if $f'$ is so.
\end{lem}
\begin{proof}First assume that $f$ is an $n$-$\Ext$-phantom morphism. Take an arbitrary object $K\in\C$. Applying the functor $\hom_{\C}(-, K)$ to the above diagram, gives us the following commutative diagram with exact rows: {\footnotesize \[\xymatrix{\Ext^{n+1}(Q, K)\ar[r]\ar[d]& \ \  \Ext^{n+1}(\syz N, K)\ar[r]\ar[d]_{\Ext^{n+1}(f', K)} \ \ & \Ext^{n+2}(N, K)\ar[d]_{\Ext^{n+2}(f, K)}\\ \Ext^{n+1}(P, K)\ar[r]& \ \ \Ext^{n+1}(\syz M, K)\ar[r]\ \ & \Ext^{n+2}(M, K).}\]}As $P, Q\in n$-$\proj\C$, $\Ext^{n+1}(Q, K)=0=\Ext^{n+1}(P, K)$. On the other hand, since $f$ is an $n$-$\Ext$-phantom morphism, one may easily see that it is an $(n+1)$-$\Ext$-phantom morphism, meaning that the right column is zero. Thus $\Ext^{n+1}(f', K)=0$, and so, $f'$  is an $n$-$\Ext$-phantom morphism. For the converse, assume that $f'$ is an $n$-$\Ext$-phantom morphism. Now one applies the functor $\hom_{\C}(K, -)$ to the former diagram, and uses the fact that $f'$ is an $(n+1)$-$\Ext$-phantom morphism to infer that $f$ is also an $n$-$\Ext$-phantom morphism. So the proof is finished.
\end{proof}

\begin{theorem}\label{sem}$\Syz:\C_{\p}\lrt\C_{\p}$ is an equivalence functor.
\end{theorem}
\begin{proof}Assume that $M$ and $N$ are two arbitrary objects of $\C_{\p}$. We would like to show that $\Syz:\hom_{\C_{\p}}(M, N)\lrt\hom_{\C_{\p}}(\Syz M, \Syz N)$ is an isomorphism. To this end, consider an object $(\bar{\ga}, \delta_N)\in\hom_{\C_{\p}}(M, N)$ with $\Syz(\bar{\ga}, \delta_N)=0$. We must show that $\bar{\ga}=0$, namely, $\ga$ is a $\p$-conflation. {Take an $\ruf$ $\ga=\delta_{N'}g$ of $\ga$, where $g:M\rt N'$ is a morphism in $\C$. As explained in Remark \ref{22},  $\Syz(\bar{\ga}, \delta_N)=(\overline{((\delta_{\syz N'}g'){b}^{-1})a}, \delta_{\syz N}),$ where $g':\syz'M\rt\syz N'$ is a morphism in $\C$ with $\delta_1g=g'\delta'_M$ and $\delta_{M}\st{a}\rt\delta''_{M}\st{b}\lf\delta'_{M}$ is an angled pair.} Since $a$ and $b$ are $n$-$\Ext$-invertible morphisms, our assumption together with \cite[Proposition 6.6]{bfss}, would imply that $\delta_{\syz N'}g'$ is a $\p$-conflation. {Thus, $g'$ will be an $n$-$\Ext$-phantom morphism, thanks to \cite[Proposition 6.2]{bfss}. Now one may apply Lemma \ref{extphantom} and deduce that the same is true for $g$, and so, another use of  \cite[Proposition 6.2]{bfss},  implies that $\ga=\delta_{N'}g$ is a $\p$-conflation. This means that the functor $\Syz$ is faithful.

Now we show that $\Syz$ is full. Assume that $(\bar{\th}, \delta_{\syz N})\in\hom_{\C_{\p}}(\Syz M, \Syz N)$ is arbitrary, with $\th:\syz^{n+1}N\rat X_n\rt\cdots\rt X_1\re\syz M$ and $\delta_{\syz N}:\syz^{n+1}N\rat P_n\rt\cdots\rt P_1\re\syz N$. Set $\th_1:=\syz^{n+1}N\rat X_n\re L$ and $\th_2:= L\rat X_{n-1}\rt\cdots\rt X_1\re\syz M$. Taking an inflation $X_n\rt Q$ with $Q\in n$-$\proj\C$, one may obtain the following commutative diagram:{\footnotesize \[\xymatrix{\th_1: \syz^{n+1}N \ar@{>->}[r]\ar@{=}[d]& \ \ X_n\ar@{->>}[r] \ \ \ar[d] & L\ar[d]_{h}\\ \th'_1:\syz^{n+1}N\ar@{>->}[r]& \ \ Q\ar@{->>}[r] \ \ & T.}\]}Consider the following co-angled pair: {\footnotesize \[\xymatrix{\syz^{n+1}N \ar@{>->}[r]\ar@{=}[d]& \ \ Q\ar@{->>}[r] \ \ & \ T \\ \syz^{n+1}N\ar@{>->}[r]\ar@{=}[d]& \ \ P_n\oplus Q\ar@{->>}[r]\ar[d]\ar[u] \ \ & T'\ar[d]_{a}\ar[u]^{b}\\ \syz^{n+1}N\ar@{>->}[r]& \ \ P_n\ar@{->>}[r]\ \ & \syz^nN.}\]}Since $a$ and $b$ are $n$-$\Ext$-invertible morphisms, applying \cite[Lemma 3.9]{bfss} gives us unit conflations $\delta_1\in\Ext^n(N, T)$ and $\delta\in\Ext^n(N, T')$ such that $\delta_N=a\delta$ and $\delta_1=b\delta$. Set $\be:=h\th_2\delta_M$, where $\delta_M: \syz M\rat P\re M$,  and $\ga:=ab^{-1}\be$. It is easy to observe that $(\bar{\ga}, \delta_N)=(\bar{\be}, \delta_1)$. Now, putting $\th':=\th'_1h\th_2$ and applying Remark \ref{22}, one may easily see that $\Syz(\bar{\ga}, \delta_N)=\Syz(\bar{\be}, \delta_1)=(\bar{\th'}, \delta_{\syz N})=(\bar{\th}, \delta_{\syz N})$. Thus the functor $\Syz$ is full.
The final issue to address is to examine the density of $\Syz$. To this end, take an arbitrary object $X\in\C$. As $\C$ is an $n$-Frobenius category, there is a unit conflation $X\rat Q_{-1}\re\syz^{-1}X$ in $\C$.  This, in particular, reveals that $X\cong\syz(\syz^{-1}X)=\Syz(\syz^{-1}X)$ in $\C_{\p}$, giving the desired result.
}
\end{proof}

The following observation, will be used in the proof of Theorem \ref{axiomtr3}.
\begin{s}\label{123}
Assume that $(\bar{\ga},\delta_N)\in\hom_{\C_{\p}}(M, N)$ is a morphism with $\ga: \syz^nN\rat X_{n-1}\rt\cdots\rt X_1\rt X_0\re M$. In view of \cite[Lemma 5.1(1)]{bfss}, we may assume that $\ga$ has a form $\syz^nN\rat P_{n-1}\rt\cdots\rt P_1\rt X_0\re M$, where each $P_i$ is $n$-projective. Breaking $\ga$, gives us two  conflations $\ga_1:\syz^nN\rat P_{n-1}\rt\cdots\rt P_1\re T$ and $\ga_2: T\rat X_0\re M$ in $\C$. Consider an $\luf$ of  $\ga_2$ as follows: \[\xymatrix{\delta'_M:\syz'M~\ar@{>->}[r]^{i}\ar[d]_{f}& Q\ar@{->>}[r]\ar[d]_{\pi}& M\ar@{=}[d]\\ \ga_2:T~\ar@{>->}[r]^{g} &X_0~\ar@{->>}[r]^{h} & M.}\]So, by taking a unit conflation $\delta_T:\syz^{n+1}N\rat P_n\rt P_{n-1}\rt\cdots\rt P_1\re T$, one has $\Syz(\bar{\ga}, \delta_N)=(\overline{((\delta_Tf) b^{-1})a}, \delta_{T})$, where $\delta_{M}\st{a}\rt\delta''_{M}\st{b}\lf\delta'_{M}$ is an angled pair.
Consider the pull-back diagram \[\xymatrix{\ga'_2:\syz T~\ar@{>->}[r]^{g'}\ar@{=}[d]& L\ar@{->>}[r]^{h'}\ar[d]_{e}& \syz'M\ar[d]_{f}\\ \delta_1: \syz T~\ar@{>->}[r]^{u} &P_1~\ar@{->>}[r]^{\pi_1} & T.}\]Combining   the preceding diagrams, yields the following commutative diagram: {\footnotesize \[\xymatrix{\ga'_2:\syz T \ar@{>->}[r]^{g'}\ar@{>->}[d]_{u}& \ \ L\ar@{->>}[r]^{h'} \ \ \ar@{>->}[d]_{[e~~ih']^t} & \ \syz' M\ar@{>->}[d]_{i}\\ P_1\ar@{>->}[r]^{[1~~0]^t}\ar@{->>}[d]_{\pi_1}& \ \ P_1\oplus Q\ar@{->>}[r]^{[0~~1]}\ar@{->>}[d]_{[g\pi_1~~\pi]} \ \ & Q\ar@{->>}[d]_{h\pi}\\ \ga_2:T\ar@{>->}[r]^{g}& \ \ X_0\ar@{->>}[r]^{h}\ \ & M.}\]}As observed just above,  $\Syz(\bar{\ga}, \delta_N)=(\overline{((\delta_2\ga'_2) b^{-1})a}, \delta_{T})$ with $\delta_2:\syz^{n+1}N\rat P_n\rt\cdots\rt P_2\re\syz T$.
\end{s}

\section{Realizing commutative diagrams in  phantom stable categories}

This section is devoted to describing commutativity of squares in phantom stable categories via commutative diagrams in $\C$. This will play a crucial role in the verification of the validity of the axiom (TR3) for $\C_{\p}$, which is stated in the next section. We begin with a useful lemma.

Assume that $M\st{(\bar{\ga},\delta_N)}\lrt N\st{(\bar{\be}, \delta_K)}\lrt K$ are two composable morphisms in $\C_{\p}$. As mentioned in  \ref{co}, $\overline{{\be}\circ{\ga}}=\overline{((\be a)b^{-1})f}$, where $\ga=\delta f$ is an $\ruf$ of $\ga$ and $[\delta_N a, \delta b]$ is a co-angled pair. The result below, reveals that this can be also computed by taking an $\luf$ of $\be$ and considering an angled pair.

\begin{lem}\label{luf}Keeping the notation above, let $\be=g\delta'_N$ be an $\luf$ of $\be$ and $[c\delta_N, d\delta'_N]$ an angled pair. Then $(\overline{\be\circ\ga}, \delta_K)=(\overline{g(d^{-1}(c\ga))}, \delta_K)$.
\end{lem}
\begin{proof}Assume that $\ga=\delta f$ is an $\ruf$ of $\ga$. Taking a co-angled pair $[\delta_Na, \delta b]$, we have that $\bar{\be}\circ\bar{\ga}=\overline{((\be a)b^{-1})f}$. Since $\be=g\delta'_N$ and $[c\delta_N, d\delta'_N]$ is an angled pair, one obtains the equalities
\[\begin{array}{lllll}
((\be a)b^{-1})f&=_{\p}((g\delta'_N a)b^{-1})f&\\ &=_{\p}g(((d^{-1}c\delta_N)
a)b^{-1})f \\ &=
_{\p}g(d^{-1}(c((\delta_Na)b^{-1})))f\\
&=_{\p}g(d^{-1}(c\delta))f=_{\p}g(d^{-1}(c\ga)),
\end{array}\]
giving the desired result.
\end{proof}

\begin{rem}\label{sq}Assume that 
\begin{equation}\label{1122}
\xymatrix{M~\ar[r]^{(\overline{\delta_Nf}, \delta_N)}\ar[d]_{(\bar{\th}, \delta_{X})} & N\ar[d]^{(\bar{\be},\delta_Y)}\\ X~\ar[r]^{(\bar{\ga}, \delta_Y)} & Y,}
\end{equation}
is a commutative square in $\C_{\p}$. By the definition of composition of morphisms, we have $(\bar{\be}, \delta_Y)\circ(\overline{\delta_Nf}, \delta_N)=(\overline{\be f}, \delta_Y)$.  Moreover, assume that $\ga=t\delta'_X$ is an $\luf$ of $\ga$, in which $t:{\syz'}^nX\rt\syz^n Y$ is a morphism in $\C$. Set  $\bar{\th'}:=\overline{b^{-1}(a\th)}$, where $\delta_X\st{a}\rt\delta''_X\st{b}\lf\delta'_X$ is an angled pair. By Lemma \ref{luf}, we get the equality $\bar{\ga}\circ\bar{\th}=\overline{t\th'}$. Indeed, $(\bar{\th}, \delta_X)=(\bar{\th'}, \delta'_X)$, and so, $(\bar{\ga}, \delta_Y)\circ(\bar{\th}, \delta_X)=(\bar{\ga}, \delta_Y)\circ(\bar{\th'}, \delta'_X)=(\overline{t\th'}, \delta_Y)$. Namely, the above square is equal to the following one: {\footnotesize \[\xymatrix{M~\ar[r]^{(\overline{\delta_{N}f},\delta_{N})}\ar[d]_{(\bar{\th'}, \delta'_{X})} & N\ar[d]^{(\bar{\be},\delta_Y)}\\ X~\ar[r]^{(\overline{t\delta'_{X}}, \delta_Y)} & Y.}\]} This means that  every commutative square of the form (\ref{1122}) in $\C_{\p}$, can be considered as the latter one.
\end{rem}

The following easy lemma plays an essential role in describing commutative squares in the phantom stable categories.
\begin{lem}\label{pp1}  Consider a commutative diagram in $\C$
\begin{equation}\label{11}
\begin{split}
{\footnotesize \xymatrix{\alpha:N \ ~\ar@{>->}[r] \ \ \ar[d]_f& \ \ X_{k-1}\ar[r] \ \ \ar[d]_{h_{k-1}}& \cdots \ar[r]& \ \ X_0\ar@{->>}[r] \ \ \ar[d]_{h_0} & \ M\ar[d]_g\\ \be:N' \ ~\ar@{>->}[r] \ \ &\ \ Y_{k-1}\ar[r] \ \ &\cdots\ar[r]& \ \ Y_0\ar@{->>}[r] \ \ & M',}}
\end{split}
\end{equation}
where $k\geq 1$. Then $f\al=\be g$. 
\end{lem}
\begin{proof}
Since the case $k=1$ is \cite[Chapter VII, Lemma 1.1]{mit}, we assume that $k\geq 2$. Take the pull-back diagram
{\footnotesize \[\xymatrix{\be g:N'\ ~\ar@{>->}[r] \ \ \ar@{=}[d]& \ \ Y_{k-1}\ar[r] \ \ \ar@{=}[d]& \cdots \ar[r]& \ \ Y_1\ar[r] \ \ \ar@{=}[d] & \ \ T\ar@{->>}[r] \ \ \ar[d] & \ M\ar[d]_g\\ \beta:N' \ ~\ar@{>->}[r] \ \ &\ \ Y_{k-1}\ar[r] \ \ &\cdots\ar[r]& \ \ Y_1\ar[r]& \ \ Y_0\ar@{->>}[r] \ \ & M'.}\]}The universal property of the pull-back yields the existence of the commutative diagram {\footnotesize \[\xymatrix{\al:N\ ~\ar@{>->}[r] \ \ \ar[d]_{f}& \ \ X_{k-1}\ar[r] \ \ \ar[d]_{h_{k-1}}& \cdots \ar[r]& \ \ X_1\ar[r] \ \ \ar[d]_{h_1} & \ \ X_0\ar@{->>}[r] \ \ \ar[d] & \ M\ar@{=}[d]\\ \be g:N' \ ~\ar@{>->}[r] \ \ &\ \ Y_{k-1}\ar[r] \ \ &\cdots\ar[r]& \ \ Y_1\ar[r]& \ \ T\ar@{->>}[r] \ \ & M.}\]}Similarly, taking the push-out of $\al$ along $f$ and considering the latter diagram, the universal property of the push-out would give us the commutative diagram
{\footnotesize \[\xymatrix{f\al:N'\ ~\ar@{>->}[r] \ \ \ar@{=}[d]& L\ar[r]\ar[d]& \ \ X_{k-2}\ar[r] \ \ \ar[d]_{h_{k-2}} & \cdots \ar[r] &\ \ X_1\ar[r] \ \ \ar[d]_{h_1} & \ \ X_0\ar@{->>}[r] \ \ \ar[d] & \ M\ar@{=}[d]\\ \be g:N' \ ~\ar@{>->}[r] \ \ &\ \ Y_{k-1}\ar[r] \ \ &\ \ Y_{k-2}\ar[r] \ \ &\cdots\ar[r]& \ \ Y_1\ar[r]& T\ar@{->>}[r] \ \ & M.}\]}Now applying \cite[Chapter VII, Proposition 3.1]{mit} yields that $f\al=\be g$, as claimed.
\end{proof}

\begin{rem}\label{remsq2} Consider the following commutative diagram in $\C$ {\footnotesize\[\xymatrix{\th: \syz^n Y~\ar@{>->}[r]\ar[d]_{t}& E_{n-1} \ar[r]\ar[d]&\cdots\ar[r]& E_0\ar@{->>}[r]\ar[d] & M \ar[d]^{f}\\ \be: \syz^n X~\ar@{>->}[r]& Z_{n-1}\ar[r] &\cdots~\ar[r]& Z_0\ar@{->>}[r] & N.}\]}According to Lemma \ref{pp1}, one has the equality $t\th=\be f$. This, in conjunction with  the definition of the composition of morphisms in $\C_{\p}$ and Lemma \ref{luf}, would give us the following commutative square in $\C_{\p}$ {\footnotesize \[\xymatrix{M~\ar[r]^{(\overline{\delta_Nf}, \delta_N)}\ar[d]_{(\bar{\th}, \delta_Y)} & N\ar[d]^{(\bar{\be}, \delta_X)}\\ Y~\ar[r]^{(\overline{t\delta_Y}, \delta_X)} & X.}\]} 
\end{rem}

In the following, we remind a characterization of  commutative squares in the stable category of a Frobenius category.

\begin{rem}\label{rr}Assume that $\C$ is a Frobenius category.  It is known that a given square  {\footnotesize \[\xymatrix{M~\ar[r]^{\bar{f}}\ar[d]_{\bar{h}} & N\ar[d]^{\bar{g}}\\ Y~\ar[r]^{\bar{e}} & X,}\]}in $\c$ is commutative if and only if there exist two morphisms $a: M\rt P$ and $b: P\rt X$ in $\C$ with $P\in\proj\C$, such that  {\footnotesize \[\xymatrix{M~\ar[r]^{[f~~a]^t}\ar[d]_{h} & N\oplus P\ar[d]^{[g~~b]}\\ Y~\ar[r]^{e} & X}\]}is a commutative square in $\C$.
\end{rem}

The next proposition is the main result of this section. It gives a description of commutative squares in phantom stable category $\C_{\p}$ in terms of commutative diagrams in $\C$. 

\begin{prop}\label{sq1}Consider the following square in $\C_{\p}$: {\footnotesize \[\xymatrix{M~\ar[r]^{(\overline{\delta_Nf}, \delta_N)}\ar[d]_{(\bar{\th}, \delta_Y)} & N\ar[d]^{(\bar{\be}, \delta_X)}\\ Y~\ar[r]^{(\overline{t\delta_Y}, \delta_X)} & X,}\]}where $t:\syz^n Y\rt\syz^n X$ is a morphism in $\C$. Then the square is commutative in $\C_{\p}$ if and only if there exists a conflation $\al\in\Ext^n(P, \syz^nX)$ with $P\in n$-$\proj\C$ and a morphism $e:M\rt P$ in $\C$ such that the following diagram in $\C$: {\footnotesize\[\xymatrix{\th: \syz^n Y~\ar@{>->}[r]\ar[d]_{t}& E_{n-1} \ar[r]\ar[d]&\cdots\ar[r]& E_0\ar@{->>}[r]\ar[d] & M \ar[d]^{[f~~e]^t}\\ \be': \syz^n X~\ar@{>->}[r]& Z_{n-1}\ar[r] &\cdots~\ar[r]& Z_0\ar@{->>}[r] & N\oplus P,}\]}is commutative, where $\be'=\nabla(\be\oplus\al)$ with $\nabla:\syz^n X\oplus\syz^n X\st{[1~~1]}\lrt\syz^n X$ .
\end{prop}
\begin{proof}Since for $n\geq 2$, the proof goes along the same lines as that of the case $n=1$, we prove the result  exclusively for $n=1$. As  the `if' part follows from Remark \ref{remsq2}, we shall prove the `only if' part. By definition, $\bar{\be}\circ\overline{\delta_N f}=\overline{\be f}$. On the other hand, applying Lemma \ref{luf}, yields that $\overline{t\delta_Y}\circ\bar{\th}=\overline{t\th}$. Since the square is commutative in $\C_{\p}$, there exists a conflation $\al:\syz X\rat L\re P$ with $P\in$1-$\proj\C$ and a morphism $e:M\rt P$ such that $t\th-\be f=\al e$. So $t\th=\be f+\al e=\nabla(\be f\oplus\al e)\Delta$, where $\nabla:\syz X\oplus\syz X\st{[1~~1]}\lrt\syz X$ and $\Delta:M\st{{{\tiny {\left[\begin{array}{ll} 1 \\ 1 \end{array} \right]}}}}\lrt M\oplus M$. Consider the following equalities: $$\nabla(\be f\oplus\al e)\Delta=\nabla((\be\oplus\al)(f\oplus e))\Delta=\nabla(\be\oplus\al)({{\tiny {\left[\begin{array}{ll} f & 0 \\ 0 & e \end{array} \right]}}}{{\tiny {\left[\begin{array}{ll} 1 \\ 1 \end{array} \right]}}})=\nabla(\be\oplus\al){{{\tiny {\left[\begin{array}{ll} f \\ e \end{array} \right]}}}},$$where the first equality follows from \cite[Chapter VII, Lemma 1.4]{mit} and  the second one is true by \cite[Chapter VII, Lemma 1.3]{mit}. 
 Namely, we have that $t\th=\be'[f~~e]^t$, where $\be'=\nabla(\be\oplus\al)$. Thus, applying Lemma \ref{pp1}  gives us the commutative diagram {\footnotesize\[\xymatrix{\th: \syz Y~\ar@{>->}[r]\ar[d]_{t}& E_1 \ar@{->>}[r]\ar[d]& M \ar[d]^{[f~~e]^t}\\ \be': \syz X~\ar@{>->}[r] &Z~\ar@{->>}[r] & N\oplus P,}\]} which completes the proof.

\end{proof}

\section{Triangulated structure on phantom stable categories}
This section contains the main result of the paper, which states that the phantom stable category  $(\C_{\p}, T)$   possesses a triangulated structure with the shift functor $\Syz$.  Triangulated categories, first introduced by Verdier \cite{ver}  in the mid 1960s,   are important structures lying at the confluence of several exciting areas of mathematics. We will  begin  by recalling the definition of a triangulated category.

\begin{dfn}
Let $\A$ be an additive category and $\T:\A\rt\A$ an auto-equivalence, which is called a shift (or translation) functor. Let $\Delta$ be a collection of sequences of morphisms in $\A$ having the form
$$
\T X\xrightarrow{g} Z\xrightarrow{h} Y\xrightarrow{f} X,
$$
which are called {\it exact triangles}. Then the triple $(\A, \T, \Delta)$ is said to be a {\it triangulated category}, if $\Delta$ is closed under isomorphisms and the following conditions are satisfied.
\begin{enumerate}

\item[(TR1)]
For any morphism $f:Y\to X$ in $\A$, there is an exact triangle
$\T X\xrightarrow{g} Z\xrightarrow{h} Y\xrightarrow{f} X$. Moreover, for any object $X\in\A$, the sequence $\T X\rt 0\rt X\st{\id}\rt X$ is an exact triangle.
\item[(TR2)](Rotation axiom)
For a given exact triangle $\T X\xrightarrow{g} Z\xrightarrow{h} Y\xrightarrow{f} X$, the sequence $\T Y\xrightarrow{-\T f}\T X\xrightarrow{g} Z\xrightarrow{h} Y$ is an exact triangle.
\item[(TR3)]
For any commutative diagram of the form
$$
\xymatrix{
\T X\ar[r]^g\ar[d]^{\T\gamma} & Z\ar[r]^h & Y\ar[r]^f\ar[d]^\beta & X\ar[d]^\gamma \\
\T X'\ar[r]^{g'} & Z'\ar[r]^{h'} & Y'\ar[r]^{f'} & X',
}
$$
where the rows are exact triangles, there is a morphism $\alpha:Z\to Z'$ in $\A$ making the completed diagram commutative.
\item[(TR4)]

For two morphisms $g:Z\to Y$ and $f:Y\to X$ in $\A$, there exists a commutative diagram
\[\xymatrix{& \T A\ar[r]\ar[d]& \T Y\ar[d]&\\ \T Y\ar[d]_{\T(f)}\ar[r]& C\ar@{=}[r]\ar[d]& C\ar[d]&\\ \T X~\ar[r]\ar@{=}[d]& B \ar[r]\ar[d]& Z\ar[r]^{fg}\ar[d]_{g} & X\ar@{=}[d]\\ \T X~\ar[r]& A\ar[r]& Y\ar[r]^{f} &X,}\]
where the last two rows and the two central columns are exact triangles.
\end{enumerate}
\end{dfn}

Inspired by the case for the stable category of a Frobenius category, we define exact triangles in  phantom stable categories as follows.

\begin{dfn}Assume that $f:M\rt N$ is a morphism in $\C$ and {$\delta_N: \syz N\rat P\re N$ is a syzygy sequence}. Consider the following pull-back diagram in $\C$:
\[\xymatrix{\syz N~\ar@{>->}[r]^{g} \ar@{=}[d]& T_f\ar@{->>}[r]^{h}\ar[d]& M\ar[d]_f\\ \syz N~\ar@{>->}[r] &P~\ar@{->>}[r]^{\pi} & N.}\]So we have the sequence $\syz N\st{T(g)}\rt T_f\st{T(h)}\rt M\st{T(f)}\rt N$ in $\C_{\p}$, which we call a standard triangle. Any sequence in $\C_{\p}$ which is isomorphic to a standard triangle, will be called an exact triangle.
\end{dfn}

\begin{lem}\label{lem22}Construction of standard triangles  is independent of the choice of syzygy sequences.
\end{lem}
\begin{proof} Assume that $f: M\rt N$ is a morphism in $\C$ and $\delta_N: \syz N\rat P\re N$, $\delta'_N:\syz'N\rat P'\re N$ are two syzygy sequences which give rise to standard triangles  $\syz N\st{T(g)}\lrt T_f\st{T(h)}\lrt M\st{T(f)}\lrt N$ and  $\syz' N\st{T(g')}\lrt T'_f\st{T(h')}\lrt M\st{T(f)}\lrt N$. Considering an  angled pair $\delta_N\st{a}\lrt\delta''_N\st{a'}\llf\delta'_N$, one may obtain the following commutative diagram:
{\footnotesize\[\xymatrix{\syz N~\ar@{>->}[r]^{g} \ar[d]_{a}& T_f\ar@{->>}[r]^{h}\ar[d]& M\ar[d]_{f}\\ \syz''N~\ar@{>->}[r] & P\oplus P'\ar@{->>}[r]& N \\ \syz'N\ar@{>->}[r]^{g'}\ar[u]^{a'} & T'_f \ar@{->>}[r]^{h'}\ar[u] & M\ar[u]^f.}\]}Taking the pull-back diagram
\[\xymatrix{\syz'' N~\ar@{>->}[r]^{g''} \ar@{=}[d]& T''_f\ar@{->>}[r]^{h''}\ar[d]& M\ar[d]_f\\ \syz'' N~\ar@{>->}[r] &P\oplus P'~\ar@{->>}[r]^{[\pi~~\pi']}& N}\]and using the universal property of pull-back diagrams, we shall get  the following commutative diagram in $\C$:{\footnotesize\[\xymatrix{\syz N~\ar[r]^{g} \ar[d]_{a}& T_f\ar[r]^{h}\ar[d]_{c}& M\ar@{=}[d]\ar[r]^{f}&N\ar@{=}[d]\\ \syz''N~\ar[r]^{g''} &T''_f~\ar[r]^{h''}& M\ar@{=}[d]\ar[r]^{f}&N\ar@{=}[d]\\ \syz'N~\ar[r]^{g'}\ar[u]^{a'} &T'_f~\ar[r]^{h'} \ar[u]^{c'}& M\ar[r]^{f}&N.}\]}One should note that since  $a, a'$ are $n$-$\Ext$-invertible morphisms, the same is true for $c$ and $c'$.  Hence applying the functor $T$ gives the desired result.
\end{proof}

Assume that $f:M\rt N$ is a morphism in $\C$ and $\syz N\st{T(g)}\rt T_f\st{T(h)}\rt M\st{T(f)}\rt N$ is the corresponding standard triangle in $\C_{\p}$. According to Lemma \ref{lem22}, this sequnce is unique up to isomorphism.  We shall  refer to this triangle,  as the triangle that $f$ is embedded.\\

As mentioned in the introduction, the stable category of a Frobenius category carries a triangulated structure, wherein the shift functor is given by the first syzygy functor $\syz$. In what follows  it will be shown that, in the similar way, the phantom stable category $\C_{\p}$ of an $n$-Frobenius category $\C$ is also triangulated, {with the shift functor $\Syz$, the syzygy functor on $\C_{\p}$. Precisely, assume that $\Delta$ is the class of all exact triangles in $\C_{\p}$. The main result of the paper asserts that $(\C_{\p}, \Syz, \Delta)$ is a triangulated category. We should emphasize that the approach to prove this result is similar to the case for the stable category of a Frobenius category.  First we quote a couple of useful facts.
\begin{obs}\label{tt}Standard triangles can be represented as one of the following forms.\\ (1) Let $f:M\rt N$ be a morphism in $\C$ and $\delta_N:\syz N\rat P\st{\pi}\re N$ a syzygy sequence. Suppose
$\syz N\st{T(g)}\lrt X\st{T(h)}\lrt M\st{T(f)}\lrt N$ is the standard triangle that $f$ is embedded. The following pull-back diagram:
{\[\xymatrix{& X\ar@{=}[r]\ar@{>->}[d]_{[1~~t]^t}& X\ar@{>->}[d]_{[h~~t]^t}\\
\syz N~\ar@{>->}[r]^{[g~~0]^t} \ar@{=}[d]& X\oplus P\ar@{->>}[r]^{h\oplus 1}\ar@{->>}[d]_{[t~~-1]}& M\oplus P\ar@{->>}[d]_{[f~~-\pi]}\\ \syz N~\ar@{>->}[r]^{s} &P~\ar@{->>}[r]^{\pi} & N,}\]gives the standard triangle $\syz N\st{T([g~~0]^t)}\lrt X\oplus P\st{T(h\oplus 1)}\lrt M\oplus P\st{T([f~~-\pi])}\lrt N$ in $\C_{\p}$.} One should note that since $[1~~t]^t$ is an inflation with $n$-projective cokernel, it is an $n$-$\Ext$-invertible morphism. Also, as $P=0$ in $\C_{\p}$ and $T([f~~-\pi])=T(f)$, the latter triangle is the same as the former one. This implies that in any standard triangle of the form $\syz N\st{T(g)}\lrt X\st{T(h)}\lrt M\st{T(f)}\lrt N$ in $\C_{\p}$, $f$ can be assumed to be a deflation in $\C$, and  there is an exact triangle $\syz N\st{T(g')}\lrt X'\st{T(h')}\lrt M'\st{T(f)}\lrt N$,
which is isomorphic to the previous standard triangle, such that $X'\st{h'}\rat M'\st{f}\re N$ is a conflation in $\C$.
 Conversely, assume that $\et:X\rat M\st{f}\re N$ is  a conflation in $\C$. Taking an $\luf$ $\et=g\delta_N$, where $\delta_N$ is a syzygy sequence of $N$ and $g:\syz N\rt X$ is a morphism in $\C$, we obtain a triangle $\syz N\st{T(g)}\lrt X\st{T(h)}\lrt M\st{T(f)}\lrt N$, which is isomorphic to the former standard triangle. {This should also be compared with the standard triangles in the stable category of a Frobenius category, see \cite[Theorem 2.6]{ha1}.} 
\\ {(2) Assume that $f:M\rt N$ is a morphism in $\C$ and $\syz N\st{T(g)}\lrt X\st{T(h)}\lrt M\st{T(f)}\lrt N$ is the standard triangle that $f$ is embedded. As declared in the previous statement, {$X\st{h}\rat M\st{f}\re N$ can be assumed to be a conflation} in $\C$. {Consider a unit conflation {$\delta_X:\syz^nX\rat Q_{n-1}\rt\cdots\rt Q_0\re X$ and a deflation $P_0\re M$. So we will get the following commutative diagram:  {\footnotesize \[\xymatrix{\delta_1:\syz X \ar@{>->}[d]_{h_1}\ar@{>->}[r]& \ \ Q_0\ar@{->>}[r] \ \ \ar@{>->}[d] & \ X\ar@{>->}[d]_{h}\\ \delta''_1:\syz M\ar@{>->}[r]\ar@{->>}[d]_{f_1}& \ \ Q_0\oplus P_0\ar@{->>}[r]\ar@{->>}[d]\ \ & M\ar@{->>}[d]_{f}\\ \delta'_1:\syz N\ar@{>->}[r]& \ \  P_0\ar@{->>}[r]\ \ & N,}\]}where all rows and columns are {exact}. Next take a deflation $P_1\re\syz M$ and get a commutative diagram 
{\footnotesize \[\xymatrix{\delta_2:\syz^2 X \ar@{>->}[d]_{h_2}\ar@{>->}[r]& \ \ Q_1\ar@{->>}[r] \ \ \ar@{>->}[d] & \ \syz X\ar@{>->}[d]_{h_1}\\ \delta''_2:\syz^2 M\ar@{>->}[r]\ar@{->>}[d]_{f_2}& \ \ Q_1\oplus P_1\ar@{->>}[r]\ar@{->>}[d]\ \ & \syz M\ar@{->>}[d]_{f_1}\\ \delta'_2:\syz^2 N\ar@{>->}[r]& \ \  P_1\ar@{->>}[r]\ \ & \syz N.}\]}Continue this manner to get  the following commutative diagram:
{\footnotesize \[\xymatrix{\delta_X:\syz^nX~\ar@{>->}[r]\ar@{>->}[d]_{h'}& Q_{n-1}\ar[r]\ar@{>->}[d]& \cdots \ar[r]& Q_0\ar@{->>}[r]\ar@{>->}[d]& X\ar@{>->}[d]_{h}\\ \delta'_M:\syz^nM~\ar@{>->}[r]\ar@{->>}[d]_{f'} &P_{n-1}\oplus Q_{n-1}\ar[r]\ar@{->>}[d]&\cdots\ar[r]& P_0\oplus Q_0\ar@{->>}[r]\ar@{->>}[d]& M\ar@{->>}[d]_{f} \\ \delta'_N:\syz^nN~\ar@{>->}[r] &P_{n-1}\ar[r]&\cdots\ar[r]& P_0\ar@{->>}[r]& N.}\]}}According to Lemma \ref{pp1}, the equalities $h'\delta_X=\delta'_Mh$ and $f'\delta'_M=\delta'_Nf$ hold.} Since $T(h)=(\overline{\delta'_Mh},\delta'_M)=(\overline{h'\delta}_X, \delta'_M)$ and $T(f)=(\overline{\delta'_Nf},\delta'_N)=(\overline{f'\delta'}_M, \delta'_N)$,} the standard triangle that $f$ is embedded, can be represented as, $\syz N\st{T(g)}\lrt X\st{(\overline{h'\delta}_X, \delta'_M)}\lrt M\st{(\overline{f'\delta'}_M, \delta'_N)}\lrt N$. In particular, {$\syz^n X\st{h'}\rat \syz^n M\st{f'}\re\syz^n N$ can be assumed to be a conflation} in $\C$.
\end{obs}

{
\begin{lem}\label{place}Every morphism in $\C_{\p}$ is placed into an exact triangle.
\end{lem}
\begin{proof}Take a morphism $(\bar{\ga},\delta_N)\in\hom_{\C_{\p}}(M, N)$. Assume that $\ga=\delta_{N'}f$ is an $\ruf$ of $\ga$, and $\syz N\st{T(g)}\rt T_f\st{T(h)}\rt M\st{T(f)}\rt N'$ is the standard triangle that $f$ is embedded. One may get the following commutative diagram in $\C_{\p}$:  \[\xymatrix{\syz N~\ar[r]^{T(g)} \ar@{=}[d]& T_f\ar[r]^{T(h)}\ar@{=}[d]& M\ar@{=}[d]\ar[r]^{T(f)}& N'\ar[d]^{(\bar{\delta}_{N'},\delta_N)}\\ \syz N~\ar[r]^{T(g)} &T_f~\ar[r]^{T(h)} & M\ar[r]^{(\bar{\ga},\delta_N)}& N.}\]In view of \cite[Corollary 7.10]{bfss}, $(\bar{\delta}_{N'},\delta_N)$ is an isomorphism, and so, the bottom row is an exact triangle, as needed.
\end{proof}

\begin{lem}\label{rot}Let $f:M\rt N$ be a morphism in $\C$ and let $\syz N\st{T(g)}\lrt T_f\st{T(h)}\lrt M\st{T(f)}\lrt N$ be the standard triangle that $f$ is embedded. Then $\syz M\st{-\Syz(T(f))}\lrt\syz N\st{T(g)}\lrt T_f\st{T(h)}\lrt M$ is an exact triangle.
\end{lem}
\begin{proof}By our hypothesis, there is a pull-back diagram \[\xymatrix{\syz N~\ar@{>->}[r]^{g} \ar@{=}[d]
& T_f\ar@{->>}[r]^{h}\ar[d]& M\ar[d]_{f}\\ \syz N~\ar@{>->}[r] &P~\ar@{->>}[r] & N.}\]Taking a deflation $u:Q\re T_f$ with $Q\in n$-$\proj\C$, gives the following commutative diagram: 
\[\xymatrix{\syz M~\ar@{>->}[r] \ar[d]_{f'}& Q\ar@{->>}[r]\ar[d]_{u}& M\ar@{=}[d]\\ \syz N~\ar@{>->}[r]^{g} &T_f~\ar@{->>}[r]^{h} & M.}\]This, in turn, yields the commutative diagram
 \[\xymatrix{\syz M~\ar@{>->}[r]^{[-f'~~i]^t} \ar@{=}[d]& \syz N\oplus Q\ar@{->>}[r]^{[g~~u]}\ar[d]_{[0~~1]}& T_f\ar[d]_{h}\\ \syz M~\ar@{>->}[r]^{i} &Q~\ar@{->>}[r] & M,}\]which gives the standard triangle  $\syz M\st{T([-f'~~i]^t)}\lrt\syz N\oplus Q\st{T([g~~u])}\lrt T_f\st{T(h)}\lrt M$ in $\C_{\p}$. As $Q\in n$-$\proj\C$,  this triangle is  the same as the triangle $\syz M\st{T(-f')}\lrt\syz N\st{T(g)}\lrt T_f\st{T(h)}\lrt M$ in $\C_{\p}$. Now it is not hard to see that $\Syz T(f)=T(f')$, and so, the proof is finished.
\end{proof}
}

In the following result, we  adapt  the proof of the axiom (TR3) for the stable category of a Frobenius category so that it applies to  the case of the phantom stable categories of $n$-Frobenius categories.
 
\begin{prop}\label{zerofrob}Let $\C$ be a Frobenius category and  let
$$
\xymatrix{
\syz N\ar[r]^{\bar{g}}\ar[d]^{\overline{\syz\beta}} & L\ar[r]^{\bar{h}} & M\ar[r]^{\bar{f}}\ar[d]^{\bar{\theta}} & N\ar[d]^{\bar{\beta}} \\
\syz Z\ar[r]^{\bar{g_1}} & X\ar[r]^{\bar{h_1}} & Y\ar[r]^{\bar{f_1}} & Z}
$$be a commutative diagram in $\c$, where the rows are exact triangles. Then there is a morphism $\eta:L\to X$ in $\C$ making the completed diagram commutative.
\end{prop}
\begin{proof}First one should note that, as stated in Observation \ref{tt}(1), the sequences {$L\st{h}\rat M\st{f}\re N$ and  $X\st{h_1}\rat Y\st{f_1}\re Z$ can be assumed to be conflations}. Applying Remark \ref{rr}, one may have the following commutative diagram  in $\C$:{\footnotesize \[\xymatrix{L\ar@{>->}[r]^{}& \ \ M\oplus P\ar@{->>}[r]^{f'} \ \ \ar[d]_{[\theta~~h']} & \ N\oplus P\ar[d]_{[\beta~~b]}\\ X\ar@{>->}[r]^{h_1}& \ \ Y\ar@{->>}[r]^{f_1}\ \ & Z,}\]}where $f'={{\tiny {\left[\begin{array}{ll} f & 0 \\ a & 1 \end{array} \right]}}}$, $P$ is  projective and $f_1h'=b$. So there is an induced morphism $\eta: L\rt X$ making the completed diagram is commutative. Now taking deflations $Q\re M\oplus P$ and $Q'\re Y$ with $Q, Q'\in\proj\C$, one gets the following commutative diagram in $\C$:
{\footnotesize \[\xymatrix{\syz N\ar@{>->}[r]\ar[d]& \ \  Q\ar@{->>}[r] \ \ \ar[d] & \ N\oplus P\ar[d]\\ \syz Z\ar@{>->}[r]& \ \ Q\oplus Q'\ar@{->>}[r]\ \ & Z.}\]}Hence, it is easily seen that  there are induced morphisms  $e, e'$ such that the following diagram is commutative in $\C$.

$$
\xymatrix{
\syz N\ar[r]^{e}\ar[d]_{{\syz\beta}} & L\ar[r]^{{}}\ar[d]_{{\eta}} & M\oplus P\ar[r]^{{f'}}\ar[d]_{{[\theta~~h']}} & N\oplus P\ar[d]_{{[\beta~~b]}} \\
\syz Z\ar[r]^{{e'}} & X\ar[r]^{{h_1}} & Y\ar[r]^{{f_1}} & Z.}
$$
This, in turn, yields the following commutative diagram in $\c$:
$$
\xymatrix{
\syz N\ar[r]^{\bar{e}}\ar[d]_{\overline{\syz\beta}} & L\ar[r]^{{}}\ar[d]_{\bar{\eta}} & M\oplus P\ar[r]^{\bar{f'}}\ar[d]_{\overline{[\theta~~h']}} & N\oplus P\ar[d]_{\overline{[\beta~~b]}} \\
\syz Z\ar[r]^{\bar{e'}} & X\ar[r]^{\bar{h_1}} & Y\ar[r]^{\bar{f_1}} & Z.}
$$
Since  $P=0$ in $\c$, one has the equalities $\bar{f}=\bar{f'}$, $\overline{[\theta~~h']}=\bar{\theta}$ and $\overline{[\beta~~b]}=\bar{\beta}$. So the proof is finished.
\end{proof}

We should point out that although the following fact is stated for 1-Frobenius categories, it remains true for each $n$-Frobenius category with $n\geq 2$.

\begin{rem}\label{1212} Assume that $\C$ is a 1-Frobenius category. Consider the following commutative diagram in $\C$:
{\footnotesize\[\xymatrix{\Syz\be:\syz^2 Z\ar[d]_{g'}\ar@{>->}[r]& \syz T~\ar@{->>}[r]\ar[d]& \syz N\ar[d]_{g}\\ \ga:\syz X~\ar@{>->}[r]\ar[d]_{h'} & K\ar@{->>}[r]\ar[d]& L\ar[d]_{h}\\ \th:\syz Y\ar@{>->}[r] \ar[d]_{f'}& E~\ar@{->>}[r]\ar[d] & M\ar[d]_{f}\\ \be:\syz Z\ar@{>->}[r] & T~\ar@{->>}[r] & N,}\]}where the rows are conflation. Applying Remark \ref{remsq2} successively, yields the following commutative diagram in $\C_{\p}$:

{\[\xymatrix{\syz N\ ~\ar[r]^{T(g)} \ \ \ar[d]_{\Syz(\bar{\be}, \delta_{Z})}& \ \ L\ar[r]^{T(h)}\ar[d]_{(\bar{\ga}, \delta_X)} \ \ & \ \ M\ar[r]^{T(f)} \ar[d]_{(\bar{\th}, \delta_{Y})}\ \ & \ N\ar[d]_{(\bar{\be}, \delta_{Z})}\\ \syz Z\ ~\ar[r]^{(\overline{g'\delta}_{\syz Z}, \delta_X)} \ \ &\ \ X\ar[r]^{(\overline{h'\delta}_X, \delta_Y)} \ \ & \ \ Y\ar[r]^{(\overline{f'\delta}_Y, \delta_Z)} \ \ & Z.}\]}
\end{rem}

Inspired by Remark \ref{1212} and modifying the proof of  Proposition \ref{zerofrob}, 
 we establish  the validity of the axiom (TR3) in the phantom stable  category $\C_{\p}$.

\begin{theorem}\label{axiomtr3}Let
\begin{equation}\label{tr3}
\xymatrix{\syz N\ ~\ar[r]^{T(g)} \ \ \ar[d]_{\Syz(\bar{\be}, \delta_{Z})}& \ \ L\ar[r]^{T(h)} \ \ & \ \ M\ar[r]^{T(f)} \ar[d]_{(\bar{\th}, \delta_{Y})}\ \ & \ N\ar[d]_{(\bar{\be}, \delta_{Z})}\\ \syz Z\ ~\ar[r]^{T(g_1)} \ \ &\ \ X\ar[r]^{T(h_1)} \ \ & \ \ Y\ar[r]^{T(f_1)} \ \ & Z,}
\end{equation}
be a commutative diagram in $\C_{\p}$, where rows are standard triangles. Then there is a morphism $L\st{(\bar{\et}, \delta_{X})}\lrt X$  making the completed diagram commutative.
\end{theorem}
\begin{proof}Let us deal only 
 with the case $n=1$, because for $n\geq 2$ the proof is similar. According to Observation \ref{tt}(2), one may find morphisms  $f'_1:\syz Y\rt\syz Z$ and $h_1':\syz X\rt\syz Y$  in $\C$ such that the equalities $T(h_1)=(\overline{h_1'\delta_X}, \delta_Y)$ and $T(f_1)=(\overline{f'_1\delta}_{Y},\delta_{Z})$ hold true. Also, considering an $\luf$ $g'_1\delta_{\syz Z}=\delta_X g_1$ of $\delta_X g_1$, where  $g_1':\syz^2Z\rt\syz X$ is a morphism in $\C$, one obtains the equality $T(g_1)=(\overline{g_1'\delta_{\syz Z}}, \delta_X)$. In particular, one gets the following commutative diagram:
\begin{equation}\label{tr3'}
\xymatrix{\syz N\ ~\ar[r]^{T(g)} \ \ \ar[d]_{\Syz(\bar{\be}, \delta_{Z})}& \ \ L\ar[r]^{T(h)} \ \ & \ \ M\ar[r]^{T(f)} \ar[d]_{(\bar{\th}, \delta_{Y})}\ \ & \ N\ar[d]_{(\bar{\be}, \delta_{Z})}\\ \syz Z\ ~\ar[r]^{(\overline{g_1'\delta}_{\syz Z}, \delta_X)} \ \ &\ \ X\ar[r]^{(\overline{h_1'\delta}_{X}, \delta_Y)} \ \ & \ \ Y\ar[r]^{(\overline{f'_1\delta}_{Y},\delta_{Z})} \ \ & Z.}
\end{equation}
Since the right square is commutative,  Proposition \ref{sq1} gives rise to the existence of the following commutative diagram in $\C$: {\footnotesize\[\xymatrix{\th: \syz Y~\ar@{>->}[r]\ar[d]_{f'_1}& E \ar@{->>}[r]\ar[d]_{h''}& M \ar[d]^{[f~~e]^t}\\ \be': \syz Z~\ar@{>->}[r] &L'~\ar@{->>}[r] & N\oplus P',}\]}for some $P'\in$1-$\proj\C$. As the bottom row  in (\ref{tr3}) is an exact triangle, by Observation \ref{tt}(1), {$X\st{h_1}\rat Y\st{f_1}\re Z$} can be assumed to be a conflation in $\C$. Now one may apply Observation \ref{tt}(2) and get the conflation $\syz X\st{h'_1}\rat\syz Y\st{f'_1}\re\syz Z$.{ In particular, one has the following diagram:\[\xymatrix{
\syz X\ar@{>->}[d]_{h_1'}& & \\ \th:\syz Y~\ar@{>->}[r]\ar@{->>}[d]_{f'_1} & E\ar@{->>}[r]\ar[d]_{h''}& M\ar[d]_{[f~~e]^t}\\ \be':\syz Z\ar@{>->}[r]^i & L'~\ar@{->>}[r]^{} & N\oplus P'.}\]
 Now taking a deflation $P\st{h'}\re L'$ with $P\in$1-$\proj\C$ and considering the preceding diagram,} we may obtain the following commutative diagram in $\C$: \[\xymatrix{\ga':\syz X\ar@{>->}[r]\ar@{>->}[d]_{h_1'}& K~\ar@{->>}[r]\ar@{>->}[d]& T\ar@{>->}[d]_{h'''}\\ \th':\syz Y~\ar@{>->}[r]\ar@{->>}[d]_{f'_1} & E\oplus P\ar@{->>}[r]^{l'}\ar@{->>}[d]_{[h''~~h']}& M\oplus P\ar@{->>}[d]_{l}\\ \be':\syz Z\ar@{>->}[r]^i & L'~\ar@{->>}[r]^{[l_1~~l_2]^t} & N\oplus P',}\]where  $l={{\tiny {\left[\begin{array}{ll} f & l_1h' \\ e & l_2h' \end{array} \right]}}}$. Consider deflations $Q_1\st{[\pi_1~~\pi_2]^t}\re E\oplus P$ and $P_1\st{\pi'_1}\re\syz Y$ with $P_1, Q_1\in$1-$\proj\C$, and get the following commutative diagram in $\C$
{\footnotesize \[\xymatrix{\be'':\syz^2Z\ar@{>->}[d]\ar@{>->}[r]& \syz L'~\ar@{->>}[r]\ar@{>->}[d]& \syz N\ar@{>->}[d]\\ \al:P_1~\ar@{>->}[r]^{[1~~0]^t}\ar@{->>}[d]_{a} & P_1\oplus Q_1\ar@{->>}[r]^{[0~~1]}\ar@{->>}[d]_{b}& Q_1\ar@{->>}[d]_{c}\\ \be':\syz Z\ar@{>->}[r]^{i} & L'~\ar@{->>}[r] & N\oplus P',}\]}where $a=f'_1\pi'_1$, $b=[b_1~~b_2]$ with $b_1=ia$, $b_2=[h''~~h'][\pi_1~~\pi_2]^t$ and $c=ll'[\pi_1~~\pi_2]^t$. Since the morphism of conflations $(a, b, c): \al\lrt\be'$ is indeed the composition of morphisms of conflatins $\al\lrt\th'\lrt\be'$, we infer that the composition of morphisms of conflations $\be''\lrt\th'\lrt\be'$ is zero. Now the universal property of kernel gives us the induced morphisms $e', e'', e'''$ in $\C$ making the following diagram commutes in $\C$.
{\footnotesize\[\xymatrix{\be'':\syz^2 Z\ar[d]_{e'}\ar@{>->}[r]& \syz L'~\ar@{->>}[r]\ar[d]_{e''}& \syz N\ar[d]_{e'''}\\ \ga':\syz X~\ar@{>->}[r]\ar[d]_{h_1'} & K\ar@{->>}[r]\ar[d]& T\ar[d]_{h'''}\\ \th':\syz Y\ar@{>->}[r] \ar[d]_{f'_1}& E\oplus P~\ar@{->>}[r]\ar[d] & M\oplus P\ar[d]_{l}\\ \be':\syz Z\ar@{>->}[r] & L'~\ar@{->>}[r] & N\oplus P'.}\]}Thus, applying Remark \ref{1212}, gives us the following commutative diagram in $\C_{\p}:$
\begin{equation}\label{tr3''}
\xymatrix{\syz N\ ~\ar[r]^{T(e''')} \ \ \ar[d]_{(\bar{\be''}, \delta_{\syz Z})}& \ \ T\ar[r]^{T(h''')}\ar[d]_{(\bar{\ga'}, \delta_X)} \ \ & \ \ M\oplus P\ar[r]^{T(l)} \ar[d]_{(\bar{\th'}, \delta_{Y})}\ \ & \ N\oplus P'\ar[d]_{(\bar{\be'}, \delta_{Z})}\\ \syz Z\ ~\ar[r]^{(\overline{e'\delta}_{\syz Z}, \delta_X)} \ \ &\ \ X\ar[r]^{(\overline{h_1'\delta}_{X}, \delta_Y)} \ \ & \ \ Y\ar[r]^{(\overline{f'_1\delta}_{Y},\delta_{Z})} \ \ & Z.}
\end{equation}
By our construction, rows are exact triangles. Since $P,P'\in$1-$\proj\C$, $T(f)=T(l)$. This, in turn, yields that the upper row is the triangle that $f$ is embedded. Particularly, $L\cong T$ in $\C_{\p}$. Also, the lower row in this diagram  is the lower triangle in the diagram (\ref{tr3}). Furthermore, using \ref{123} implies that $(\bar{\be''}, \delta_{\syz Z})=\Syz(\bar{\be'}, \delta_Z)$.  Another use of the fact that $P,P'\in 1$-$\proj\C$, yields that $\bar{\th'}=\bar{\th}$ and $\bar{\be'}=\bar{\be}$. Hence, $(\bar{\ga'}, \delta_X)$ is the desired morphism, and so, the proof is finished.

\end{proof}

The following pair of results will be used in verifying the validity of the axiom (TR4).

\begin{prop}\label{55}
 Let  {\footnotesize \[\xymatrix{\ga:X \ar[d]_{e_1}\ar@{>->}[r]^{h}& \ \ M\ar@{->>}[r]^{f} \ \ \ar[d]_{e_2} & \ N\ar[d]_{e_3}\\ \ga':X'\ar@{>->}[r]^{h'}& \ \ M'\ar@{->>}[r]^{f'}\ \ & N'}\]}be a commutative diagram in $\C$, where the rows are conflations. Then there is an induced morphism of exact triangles \[\xymatrix{\syz N~\ar[r]^{T(g)} \ar[d]_{\Syz T( e_3)}& X\ar[r]^{T(h)}\ar[d]_{T(e_1)}& M\ar[d]_{T(e_2)}\ar[r]^{T(f)}& N\ar[d]_{T(e_3)}\\ \syz N'~\ar[r]^{T(g')} &X'~\ar[r]^{T(h')} & M'\ar[r]^{T(f')}& N'.}\]
 \end{prop}
 \begin{proof}Consider an $\luf$ of $\ga$ as follows:
 {\footnotesize \[\xymatrix{\delta_N:\syz N \ar[d]_{g}\ar@{>->}[r]& \ \ Q\ar@{->>}[r] \ \ \ar[d] & \ N\ar@{=}[d]\\ \ga:X\ar@{>->}[r]^{h}& \ \ M\ar@{->>}[r]^{f}\ \ & N.}\]}So  taking  a deflation $P\re M'$ with $P\in n$-$\proj\C$, one gets the commutative diagram {\footnotesize{\[\xymatrix{&\syz N \ar@{>->}[rr]\ar[dd]_{g}\ar[ld]_{e'_3}&&Q\ar@{->>}[rr]\ar[dd]\ar[dl]&&N\ar[dl]_{e_3}\ar@{=}[dd]\\ \syz N'\ar@{>->}[rr]\ar[dd]_{g'}&&Q\oplus P\ar@{->>}[rr]\ar[dd]&& N'\ar@{=}[dd]\\ &\ga:X\ar@{>->}[rr]\ar[dl]_{e_1}&&M\ar@{->>}[rr]\ar[dl]_{e_2}&& N\ar[dl]_{e_3}\\ \ga':X'\ar@{>->}[rr]&& M'\ar@{->>}[rr]&&N'.}\]}}This, in turn, gives us the following commutative diagram in $\C$:
\[\xymatrix{\syz N~\ar[r]^{g} \ar[d]_{e'_3}& X\ar[r]^{h}\ar[d]_{e_1}& M\ar[d]_{e_2}\ar[r]^{f}& N\ar[d]_{e_3}\\ \syz N'~\ar[r]^{g'} &X'~\ar[r]^{h'} & M'\ar[r]^{f'}& N'.}\]Now applying the functor $T$ to the latter diagram and using the fact that $\Syz T(e_3)=T(e'_3)$, would give us the desired result.
\end{proof}

\begin{prop}\label{tr4}Let $A\st{(\bar{\th},\delta_B)}\lrt B\st{(\bar{\et},\delta_C)}\lrt C$ be a composition of morphisms in $\C_{\p}$. Then there exist objects $B', C'$ and morphisms $A\st{h}\rt B'\st{f}\rt C'$ in $\C$ such that {\footnotesize\[\xymatrix{ A~\ar[r]^{(\bar{\th},\delta_B)}\ar@{=}[d]& B \ar[r]^{(\bar{\et},\delta_C)}\ar[d]_{}& C \ar[d]_{}\\ A~\ar[r]^{T(h)} &B'~\ar[r]^{T(f)} & C',}\]} is a commutative diagram in $\C_{\p}$ in which the columns are isomorphisms.
\end{prop}
\begin{proof}

Take an $\ruf$ $\th=\delta_{B'}h$ of $\th$, where $h:A\rt B'$ is a morphism in $\C$ and consider a co-angled pair $\delta_B\st{a}\lf\delta_{B''}\st{b}\rt\delta_{B'}$, to get the equalities $(\bar{\th},\delta_B)=(\overline{\delta_{B'}h},\delta_B)=(\overline{(\delta_Ba)b^{-1}h},\delta_B)$. Setting $\bar{\et'}:=\overline{(\et a)b^{-1}}$, we have the morphism $B'\st{(\bar{\et'}, \delta_C)}\lrt C$ in $\C_{\p}$. Assume that $\et'=\delta_{C'}f$ is an $\ruf$ of $\et'$, where $f:B'\rt C'$ is a morphism in $\C$. Now, we will get the following commutative diagram in $\C_{\p}$:

{\footnotesize\[\xymatrix{ A~\ar[r]^{(\bar{\th},\delta_B)}\ar@{=}[d]& B \ar[r]^{(\bar{\et},\delta_C)}\ar[d]_{(\bar{\delta}_B,\delta_{B'})}& C \ar[d]_{(\bar{\delta}_C,\delta_{C'})}\\ A~\ar[r]^{T(h)} &B'~\ar[r]^{T(f)} & C',}\]}in which the all columns are isomorphisms, thanks to \cite[Corollary 7.10]{bfss}. Thus the proof is finished.
\end{proof}

Now we are ready to prove the main result of this paper.
{
\begin{theorem}\label{lt}The triple $(\C_{\p}, \Syz, \Delta)$ is a triangulated category.
\end{theorem}
\begin{proof}
(TR1): By Lemma \ref{place}, any morphism in $\C_{\p}$ is placed into a standard triangle. Moreover, for a given object $M\in\C$, we may consider a syzygy sequence $\delta_M:\syz M\st{g}\rt P\st{h}\rt M$, and taking the pull-back of $\delta_M$ along the morphism $1_M$, gives rise to the standard triangle  $\syz M\st{T(g)}\lrt P\st{T(h)} \lrt M \st{T(1_M)}\lrt M$ in $\C_{\p}$. It should be noted that $P=0$ in $\C_{\p}$.

(TR2): We must show that any rotation of an exact triangle, is again an exact triangle. In this direction, it suffices to consider standard triangles. Assume that $\syz N\st{T(g)}\rt T_f\st{T(h)}\rt M\st{T(f)}\rt N$ is a standard triangle. According to Lemma \ref{rot}, $\syz M\st{-\Syz(T(f))}\rt\syz N\st{T(g)}\rt T_f\st{T(h)}\rt M$  is an exact triangle, as needed.

(TR3): The validity of the axiom (TR3), follows from Theorem \ref{axiomtr3}.

(TR4): Assume that $A\st{(\bar{\th}, \delta_B)}\lrt B\st{(\bar{\et}, \delta_C)}\lrt C$ is a composition of morphisms in $\C_{\p}$. {By Proposition \ref{tr4}, we may replace morphisms $(\bar{\th}, \delta_B)$ and $(\bar{\et}, \delta_C)$ with $T(f)$ and $T(g)$, respectively, for some morphisms  $f:A\rt B$ and $g:B\rt C$ in $\C$. Suppose that $\syz B\st{T(f'')}\lrt E\st{T(f')}\lrt A\st{T(f)}\lrt B$,
$\syz C\st{T(h'')}\lrt F\st{T(h')}\lrt A\st{T(gf)}\lrt C$ and $\syz C\st{T(g'')}\lrt D\st{T(g')}\lrt B\st{T(g)}\lrt C$ are corresponding standard triangles. We would like to find  an exact triangle $\syz D\lrt E\lrt F\lrt D$ which commutes the following diagram in $\C_{\p}$.

\begin{equation}\label{trr}
\xymatrix{& \syz D\ar[r]\ar[d]& \syz B\ar[d]_{T(f'')} &\\ \syz B\ar[d]_{\Syz T(g)}\ar[r]& E\ar@{=}[r]\ar[d]& E\ar[d]_{T(f')} &\\ \syz C~\ar[r]^{T(h'')}\ar@{=}[d]& F \ar[r]^{T(h')}\ar[d]& A\ar[r]^{T(gf)}\ar[d]_{T(f)} & C\ar@{=}[d]\\ \syz C~\ar[r]^{T(g'')}& D\ar[r]^{T(g')}& B\ar[r]^{T(g)} &C.}
\end{equation}

 By our hypothesis, the triangle in the bottom row is arisen from the following pull-back diagram:}

  \[\xymatrix{\syz C~\ar@{>->}[r]^{g''} \ar@{=}[d]& D\ar@{->>}[r]^{g'}\ar[d]_{\varphi}& B\ar[d]_g\\ \delta_C:\syz C~\ar@{>->}[r] &P~\ar@{->>}[r]^{\pi} & C,}\]which gives us the conflation $D\st{[g'~~\varphi]^t}\rat B\oplus P\st{[-g~~\pi]}\re C$. Now taking a deflation $Q'\st{[\phi_1~~\phi_2]^t}\re B\oplus P$ with $Q'\in n$-$\proj\C$,  one may obtain the following commutative diagram in $\C$
 \begin{equation}\label{trtr}
\xymatrix{E'~\ar@{=}[r] \ar@{>->}[d]_{k}& E'\ar@{>->}[d]_{f'_1}& \\ F'~\ar@{>->}[r]^{h'_1} \ar@{->>}[d]_{l'}& A\oplus Q'\ar@{->>}[r]^{\psi}\ar@{->>}[d]_{l}& C\ar@{=}[d]\\ D~\ar@{>->}[r]^{[g'~~\varphi]^t} &B\oplus P~\ar@{->>}[r]^{[-g~~\pi]} & C,}
\end{equation}
{{where $l={{\tiny {\left[\begin{array}{ll} f & \phi_1 \\ 0 & {\phi_2} \end{array} \right]}}}$ and $\psi=[-gf~~~~\pi\phi_2-g\phi_1]$. Consider the following commutative diagram in $\C:$ 
\begin{equation}\label{ttrr}
\xymatrix{ E'~\ar@{>->}[r]^{f'_1} \ar[d]_{s}& A\oplus Q'\ar@{->>}[r]^{l}\ar@{=}[d]& B\oplus P\ar[d]_{[-g~~\pi]}\\ F'~\ar@{>->}[r]^{h'_1} &A\oplus Q'~\ar@{->>}[r]^{\psi} & C,}
\end{equation}
in which $s$ is an induced map. One should note that since $h'_1s=h'_1k$, and $h'_1$ is a monomorphism, we have that $k=s$. Now applying  Proposition \ref{55} to \ref{trtr} and \ref{ttrr}, gives us  the following commutative diagram in $\C_{\p}$:
\begin{equation}\label{tr4'}
\xymatrix{& \syz D\ar[r]\ar[d]& \syz B\ar[d] &\\ \syz B\ar[d]_{\Syz T(g)}\ar[r]& E'\ar@{=}[r]\ar[d]_{T(k)}& E'\ar[d]_{T(f'_1)} &\\ \syz C~\ar[r]\ar@{=}[d]& F' \ar[r]^{T(h'_1)}\ar[d]_{T(l')}& A\oplus Q'\ar[r]^{T(\psi)}\ar[d]_{T(l)} & C\ar@{=}[d]\\ \syz C~\ar[r]& D\ar[r]^{T([g'~~\varphi]^t)}& B\oplus P\ar[r]^{T([-g~~\pi])} &C,}
\end{equation}
}
where the rows and columns are exact triangles. Since $n$-projectives coincide with the zero objects in $\C_{\p}$ and  $T(\psi)=-T(gf)$, the exact triangles $\syz C\rt F'\st{T(h'_1)}\rt A\oplus Q'\st{T(\psi)}\rt C$ and $\syz C\st{T(h'')}\rt F\st{T(h')}\rt A\st{-T(gf)}\rt C$ are isomorphic. Also, as $T(l)=T(f)$, the exact triangles $\syz B\rt E'\st{T(f'_1)}\rt A\oplus Q'\st{T(l)}\rt B\oplus P$ and $\syz B\st{T(f'')}\rt E\st{T(f')}\rt A\st{T(f)}\rt B$ are also isomorphic, as well. These yield that the exact triangle $\syz D\rt E'\st{T(k)}\rt F'\st{T(l')}\rt D$ is the desired triangle, which appears  as the middle column in the diagram (\ref{trr}). } So the proof is completed.

\end{proof}


Assume that $\C$ is a Frobenius category. As mentioned before, it can be considered as an $n$-Frobenius category, for each integer $n\geq 1$. Suppose that $\c$ is the stable category of $\C$ modulo projectives and $F:\C\rt \c$ is the additive quotient functor. According to  Remark \ref{0frob}, $(\c, F)$ is the stabilization of  $\C$ as a 0-Frobenius category, which is known to be triangulated. On the other hand, assume that $(\C_{\p}, T)$ is the stabilization of $\C$, considered as an $n$-Frobenius category.  The following result asserts that these two  triangulated structures are equivalent. First we include a lemma.

 \begin{lem}\label{01}Let $\C$ be a Frobenius category. Then for any object $N\in\C$ and any integer $n\geq 1$, the subfunctor $\p$ of $\Ext^n(-, \syz^nN)$ vanishes. In particular, $\Ext^n(-, \syz^nN)/\p\cong\Ext^n(-, \syz^nN)$.
\end{lem}
\begin{proof}{Take an arbitrary object $M\in\C$. We would like to show that a given $\p$-conflation $\th\in\Ext^n(M, \syz^nN)$ is zero.} By the hypothesis, there exist a conflation $\ga\in\Ext^n(P, \syz^nN)$ with $P\in n$-$\proj\C$, and a morphism $f:M\rt P$ in $\C$ such that $\th=\ga f$. Since $P\in n$-$\proj\C$, it is easily seen that $P$ is a projective object of $\C$, as well. Considering $\ga$ as the conflation $\syz^nN\rat L_{n-1}\rt\cdots\rt L_0\st{\pi}\re P$, we infer that $\pi$ is a split epimorphism. This, in turn,  implies that $\ga$  is zero as an object of $\Ext^n(P, \syz^nN)$, and so, $\th$ will be a zero object of $\Ext^n(M, \syz^nN)$. Thus the proof is finished.
\end{proof}

\begin{theorem}\label{thm}Let $\C$ be a Frobenius category. Then two  triangulated categories $(\c, F)$ and $(\C_{\p}, T)$ are equivalent.
\end{theorem}
\begin{proof}Take a morphism $f$ in $\C$ that factors through a projective object. Since each projective object is $n$-projective, we have that $T(f)=0$ in $\C_{\p}$. Universal property of $(\c, F)$ gives us a unique additive covariant functor $F':\c\rt\C_{\p}$ such that $T=F'F$. Assume that $f:M\rt N$ is an $n$-$\Ext$-phantom morphism in $\C$. So, for a given unit conflation $\delta_N\in\U_n(N)$, $\delta_Nf$ is a $\p$-conflation, see \cite[Proposition 6.2]{bfss}. Hence, invoking Lemma \ref{01}, enables us to infer that $f$ factors through an $n$-projective (and so, projective) object, implying that $F(f)=0$ in $\c$. Finally, assume that $f$ is an $n$-$\Ext$-invertible morphism. We shall prove that $F(f)$ is an isomorphism. To this end,  as $f$ is an $n$-$\Ext$-invertible morphism, by \cite[Lemma 3.3]{bfss}, there exists a conflation $M\st{[f~~l]^t}\rat N\oplus P\re Q$ in $\C$, where $P, Q\in n$-$\proj\C$. Since each $n$-projective object is indeed a projective object, this sequence is split, and then, we get the isomorphism $M\oplus Q\st{{{\tiny {\left[\begin{array}{ll} f & g_1 \\ l & g_2 \end{array} \right]}}}}\lrt N\oplus P$ in $\C$. This would  imply that $F(f)$ is an isomorphism. Thus the universal property of the phantom stable category $(\C_{\p}, T)$ gives rise to the existence of a unique additive covariant functor $T':\C_{\p}\rt \c$ with $T'T=F$. Since by the universal property, the functors $F'$ and $T'$ are uniquely determined up to natural isomorphism, we have $T'F'\cong 1_{\c}$ and  $F'T'\cong 1_{\C_{\p}}$. It only remains to show that the functors $F'$ and $T'$ are triangulated. As these functors are identity on objects and $n$-projectives and projectives are the same, one may deduce that the functors $F'$ and $T'$ commute with the shift functors.
Now we show that these functors send standard triangles to standard triangles. To this end, for a given object $Z\in\C$, take a unit conflation $\delta_Z$. By the construction of $F'$, we have $F'(\bar{1}_Z)=(\bar{\delta}_Z, \delta_Z)$. Moreover, for a given morphism $\bar{f}\in\hom_{\c}(M, Z)$, one has $F'(\bar{f})=(\overline{\delta_Zf}, \delta_Z)$. Next assume that $f:M\rt N$ is a morphism in $\c$. Consider the following pull-back diagram in $\C$: \[\xymatrix{\syz N~\ar@{>->}[r]^{g} \ar@{=}[d]& T_f\ar@{->>}[r]^{h}\ar[d]& M\ar[d]_f\\ \syz N~\ar@{>->}[r] &P~\ar@{->>}[r]^{\pi} & N.}\]This gives us the standard triangle $\syz N\st{\bar{g}}\rt T_f\st{\bar{h}}\rt M\st{\bar{f}}\rt N$  in $\c$. Applying the functor $F'$ yeilds the sequemce $\syz N\st{F'(\bar{g})}\lrt T_f\st{F'(\bar{h})}\lrt M\st{F'(\bar{f})}\lrt N$  in $\C_{\p}$. This is indeed the standard triangle $\syz N\st{T(g)}\lrt T_f\st{T(h)}\lrt M\st{T(f)}\lrt N$ that $f$ is embeded. This means that the functor $F'$ sends each standard triangle in $\c$ to standard triangle in $\C_{\p}$. Similarly, one may see that the functor $T'$ also sends standard triangles in $\C_{\p}$ to standard triangles in $\c$. Thus the proof is finished.
\end{proof}

 

{\bf{Acknowledgements.}} The authors are grateful to the referee for reading the paper very carefully and providing useful suggestions kindly and patiently.



\end{document}